\documentclass{amsart}

\usepackage{amsrefs}

\usepackage{amssymb}

\usepackage[all,cmtip]{xy}

\usepackage{graphicx}
\usepackage{subcaption}
\usepackage{xcolor}
\usepackage{comment}
\usepackage{enumitem}
\usepackage{subcaption}
\usepackage{float}

\newcommand{\R}{\mathbb{R}}
\newcommand{\C}{\mathbb{C}}
\newcommand{\Z}{\mathbb{Z}}
\newcommand{\N}{{\mathbb N}}
\newcommand{\D}{{\mathbb D}}

 \newcommand{\SLE}{{\rm SLE}}

\numberwithin{equation}{section}

\theoremstyle{plain} 
\newtheorem{thm}[equation]{Theorem}
\newtheorem{cor}[equation]{Corollary}
\newtheorem{lem}[equation]{Lemma}
\newtheorem{prop}[equation]{Proposition}

\theoremstyle{definition}
\newtheorem{defn}[equation]{Definition}

\theoremstyle{remark}
\newtheorem{rem}[equation]{Remark}

\begin{document}

\title{GENERALIZED INTEGRAL MEANS SPECTRUM OF SLE}

\author{Ho Xuan Hieu}

\address{Institut de Math\'ematiques de Bordeaux UMR5251\\
Universit\'e de Bordeaux\\
352, cours de la Lib\'eration - F 33 405 TALENCE}


\email{xuan-hieu.ho@math.u-bordeaux.fr}


\date{\today}

\begin{abstract}
	We extend the validity domain of the conjecture about the average generalized integral means spectrum of whole--plane SLE introduced in \cite{DHLZ2015arXiv,DHLZ2018}. Thence we improve the results obtained in \cite{DHLZ2015arXiv,DHLZ2018} on the average generalized integral means spectrum of the m-fold transform associated to the whole--plane SLE map. An extension to a new interval of parameter for the usual average integral means spectrum of the interior whole--plane SLE is also given.
 \end{abstract}

\maketitle


 \section{Introduction}
 The {\it Loewner chains} have become well known since they were introduced for the first time by Charles Loewner in 1923 to prove the {\it Bieberbach's conjecture} for $n=3$ \cite{lowner1923}. This famous conjecture, that is nowadays a classical theorem in complex analysis and theory of univalent functions, states that if $f$ is a function in the class $\mathcal{S}=\{f: \text{ holomorphic and injective in } \D, f(0)=0,f'(0)=1\}$ and $f$ has a Taylor series expansion $f(z)=z+\sum_{n=2}^{+\infty} a_n z^n$ then
 \begin{equation}
 \vert a_n \vert\leq n.
 \end{equation}
 Ludwig Bieberbach proved that $\vert a_2 \vert\leq2$ and stated the conjecture in 1916 \cite{bieberbach1916}. While the case $n=2$ can be obtained as a consequence of the Area theorem, the proof for $n=3$ by Loewner was based on Loewner theory. A Loewner chain $(f_t)_{t\geq 0}$ is a family of univalent functions defined as the solution to the so-called Loewner equation
 \begin{equation}\label{eq_loewner}
 \frac{\partial f_t(z)}{\partial t}=z\frac{\partial f_t(z)}{\partial z}\frac{\lambda(t)+z}{\lambda(t)-z},
 \end{equation}
 where $\lambda:[0,+\infty)\rightarrow \partial\D$ is a continuous function.\\
 
 In 1999, Odded Schramm revived Loewner chain by introducing a stochastic version, the now called {\it Schramm--Loewner Evolution}. The idea is to consider the Loewner equation \eqref{eq_loewner} with a Brownian driving function  $\lambda(t):=e^{i\sqrt{\kappa}B_t}$ instead of a deterministic one. Here $B_t$ is the standard one--dimensional Brownian motion and $\kappa$ is a nonnegative real parameter. With this setting, he showed that $\SLE$ with $\kappa=2$ is a possible candidate for the scaling limits of loop-erased random walks and uniform spanning trees \cite{Schramm2000}. Since then, $\SLE$ has been intensively studied by mathematicians and physicists. It has been proved or conjectured to be the scaling limit of numerous lattice models in statistical physics.\\

 In this paper, we study the {\it average generalized integral means spectrum} of whole--plane SLE (see Section \ref{section intro SLE} for a precise definition of whole--plane--SLE). Let us first recall that the whole--plane SLE map at time 0 $(f=f_0)$ belongs to the class $\mathcal{S}$. For a function $f\in\mathcal{S}$, the {\it integral means spectrum} is defined as
 \begin{equation}\label{eq def int means spec}
 \beta_f(p):=\limsup_{r\to 1^-}\frac{\log\int_{r\partial\D}\vert f'(z) \vert^p \vert dz \vert}{-\log(1-r)}, \,\,p\in \R.
 \end{equation}
 To take care of the possible unboundedness of the image domain, we define the {\it generalized integral mean spectrum} as
  	\begin{equation}\label{eq def generalized int means spec}
  \beta_f(p,q):=\limsup_{r\to 1^-}\frac{\log\int_{r\partial\D}\vert z \vert^q\frac{\vert f'(z) \vert^p}{\vert f(z) \vert^q} \vert dz \vert}{-\log(1-r)}, \,\,(p,q)\in\R^2.
  \end{equation}
  In the case $f$ is random, we define the generalized integral means spectrum in expectation as
  \begin{equation}
 \beta(p,q):=\limsup_{r\to 1^-}\frac{\log\int_{r\partial\D}\mathbb{E}\bigg(\vert z \vert^q\frac{\vert f'(z) \vert^p}{\vert f(z) \vert^q}\bigg) \vert dz \vert}{-\log(1-r)}, \,\,(p,q)\in\R^2.
 \end{equation}
 The goal of this paper is to study this spectrum for time 0 whole--plane SLE map.
This generalized spectrum was introduced for the first time in \cite{DHLZ2018} to give a unified framework of various spectra in the SLE context. Namely, the case $q=0$ corresponds to the usual average integral means spectrum of the interior whole--plane SLE
 \begin{equation}
 \beta_f(p):=\limsup_{r\to 1^-}\frac{\log\int_{r\partial\D}\mathbb{E}(\vert f'(z) \vert^p) \vert dz \vert}{-\log(1-r)}, \label{average int means spec}
 \end{equation}
 while the case $q=2p$ corresponds to the usual average integral means spectrum of the bounded whole--plane SLE studied by Beliaev and Smirnov in \cite{BS2009}. This version of whole--plane SLE related to the interior one by an inversion symmetry
 \begin{equation}\label{eq inter-exter SLE relation}
 \hat{f}_t(\zeta)\stackrel{\rm (law)}{=}\frac{1}{f_t(\frac{1}{\zeta})}, \hspace{0.3cm}\zeta\in\Delta=\C\setminus\D.
 \end{equation}
 Moreover, the generalized spectrum $\beta(p,q_m)$ where $q_m=q_m(p,q)=(1-1/m)p+q/m$ also give the average generalized integral means spectrum of the SLE m--fold maps that are defined as $f^{[m]}: z \mapsto [f(z^m)]^{1/m}$ for $m\in \N\setminus\{0\}$ and $f^{[m]}(\zeta)=1/f^{[-m]}(1/\zeta)$ for $m\in \Z\setminus \N$.\\
 
The first rigorous results on the values of the average integral means spectrum of SLE were obtained by Beliaev and Smirnov \cite{BS2009}. In that work, they studied the (bounded) exterior whole--plane SLE which is corresponding to the case $q=2p$ of the generalized spectrum (see Section \ref{section intro SLE} for a precise definition of exterior whole--plane SLE). They used a martingale argument and a kind of Markov property of SLE to derive a PDE, the now called Beliaev--Smirnov equation, obeyed by the derivative moment $\mathbb{E}(\vert \hat{f}'(z) \vert^p)$, where $\hat{f}=\hat{f}_0$ is the SLE map at time 0. They then applied the maximum principle to estimate the integral means of  the solution on the circles $\vert z \vert=r$, for all $r$ sufficiently closed to $1$, by those of a sub--solution and a super--solution of the above equation. By constructing appropriate sub--solutions and super--solutions for the equation of $\mathbb{E}(\vert \hat{f}'(z) \vert^p)$, they showed that the spectrum has three phases
\begin{equation}\label{eq 3 phase spectrum BS}
\beta(p,2p)=\left\{
\begin{array}{ll}
\beta_{tip}(p)& p\leq p'_0(\kappa)\\
\beta_0(p) & p'_0(\kappa)<p\leq p_0(\kappa)\\
\beta_{lin}(p) & p>p_0(\kappa)\\
\end{array} 
\right.
\end{equation}
where
\begin{align}
	\beta_{tip}(p)&=-p-1+\frac{1}{4}(4+\kappa-\sqrt{4+\kappa)^2-8\kappa p}),\\
	\beta_{0}(p)&=-p+\frac{4+\kappa}{4\kappa}(4+\kappa-\sqrt{4+\kappa)^2-8\kappa p}),\\
	\beta_{lin}(p)&=p-\frac{(4+\kappa)^2}{16\kappa}
\end{align}
and
\begin{align}
	p'_0(\kappa)&=-1-3\kappa/8,\\
	p_0(\kappa)&=3(4+\kappa)^2/32\kappa.
\end{align}
The method in \cite{BDZ2016,BDZ2017,BS2009} has been developed by Duplantier, Nguyen, Nguyen, Zinsmeister \cite{DNNZ2011,DNNZ2015} for the unbounded version of whole--plane SLE. The authors discovered a new phase of the average integral means spectrum that appears before the phase $\beta_{lin}$ in \eqref{eq 3 phase spectrum BS}. Namely, they obtained
\begin{equation}
\beta(p,0)=\left\{
\begin{array}{ll}
\beta_{tip}(p)& p\leq p'_0(\kappa)\\
\beta_0(p) & p'_0(\kappa)<p\leq p^*(\kappa)\\
\beta_{1}(p,0) & p^*(\kappa)<p<\min(\hat{p}(\kappa),p(\kappa)),\\
\end{array} 
\right.
\end{equation}
where
\begin{equation}
	\beta_{1}(p)=3p-\frac{1}{2}-\frac{1}{2}\sqrt{1+2\kappa p}
\end{equation}
and
\begin{align}
	p^*(\kappa)&=(\sqrt{2(4+\kappa)^2+4}-6)(\sqrt{2(4+\kappa)^2+4}+2)/32\kappa,\\
	\hat{p}(\kappa)&=1+\kappa/2,\label{p hat}\\
	p(\kappa)&=(6+\kappa)(2+\kappa)/8\kappa.
\end{align}
The {\it bulk spectrum} $\beta_{0}$ is due to the singularity of the derivative moment around the generic part of the boundary of the SLE image domain. This spectrum has been predicted earlier by Duplantier \cite{dup2000}. The {\it tip spectrum} $\beta_{tip}$ is due to the singularity at the growing tip of SLE trace and predicted by Hasting \cite{has2002}. The {\it linear spectrum} $\beta_{lin}$ is due to the linear prolongation by convexity of the generalized spectrum. The appearance of the phase $\beta_1$ is due to the unboundedness of the interior whole--plane SLE map $f$ (and can be thought as coming from the tip at infinity). Let us notice that this new phase of spectrum has also been discovered by Loutsenko and Yermolayeva in \cite{Lou2013} where they used another method to compute the average integral means spectrum of the interior whole--plane SLE and by Beliaev, Duplantier and Zinsmeister in \cite{BDZ2017} where they filled a gap in \cite{BS2009}.\\

The methods in \cite{DNNZ2015} were again used to study the generalized integral means spectrum of whole--plane SLE in \cite{DHLZ2018}. They led the authors to introduce a generalization of the above spectrum function $\beta_1$ as
\begin{equation}
	\beta_{1}(p,q):=3p-2q-\frac{1}{2}-\frac{1}{2}\sqrt{1+2\kappa(p-q)}.
\end{equation}
In the sequel, $\beta_1$ will always stand for this generalized function. The authors conjectured that the generalized spectrum \eqref{eq def generalized int means spec} continuously takes four phases $\beta_{tip},\beta_{0},\beta_{lin},\beta_1$ on four domains of the $(p,q)$--plane separated by five transition lines. To be precise, let us first introduce the subsets of parameter $(p,q)-$plane where these spectrum functions coincide. Let $\mathcal{R}$ and $\mathcal{G}$ the parabolas, so-called \textcolor{red}{ red} and \textcolor{green}{ green}, respectively defined by
\begin{equation}
\begin{split}
p&=p_\mathcal{R}(\gamma):=(2+\frac{\kappa}{2})\gamma-\frac{\kappa}{2}\gamma^2,\\
q&=q_\mathcal{R}(\gamma):=(3+\frac{\kappa}{2})\gamma-\kappa\gamma^2, \gamma\in\R
\end{split}
\end{equation}
and
\begin{equation}
\begin{split}
p&=p_\mathcal{G}(\gamma'):=\frac{(4+\kappa)^2}{8\kappa}-\frac{\kappa}{2}\gamma'^2,\\
q&=q_\mathcal{G}(\gamma'):=\frac{(4+\kappa)^2}{8\kappa}+\gamma'-\kappa\gamma'^2, \gamma'\in\R.
\end{split}
\end{equation}
The intersection of the red and green parabolas is located at the two points
\begin{align}
	P_0: p_0&=p_0(\kappa)=\frac{3(4+\kappa)^2}{32\kappa}, q_0=\frac{(4+\kappa)(8+\kappa)}{16\kappa},\\
	P_1: p_1&=\frac{(8+\kappa)(8+3\kappa)}{32\kappa},  q_0=\frac{(4+\kappa)(8+\kappa)}{16\kappa}.
\end{align}
The spectrum functions  $\beta_{0}$ and $\beta_1$ coincide on the part of $\mathcal{R}$ corresponding to the interval $[1/\kappa,2/\kappa+1/2]$ of the parameter $\gamma$ and on the part of $\mathcal{G}$ corresponding to the interval $[1/\kappa,+\infty)$ of the parameter $\gamma'$. We call these sets the middle part of $\mathcal{R}$ and the left branch of $\mathcal{G}$.
The spectrum function $\beta_{0}$ respectively coincides with $\beta_{tip}$ and $\beta_{lin}$ on the vertical lines
\begin{align}
	D'_0:&=\bigg\{(p,q): p=-1-\frac{3\kappa}{8}\bigg\},\\
	D_0:&=\bigg\{(p,q): p=p_0=\frac{3(4+\kappa)^2}{32\kappa}\bigg\}.	
\end{align}
Note that the vertical line $D_0$ also passes through the intersection point $P_0$ of the red parabola and the green parabola.
The spectrum functions $\beta_{lin}$ and $\beta_{1}$ coincide on the line of slope 1
\begin{equation}
D_1:=\bigg\{(p,q): q-p=q_0-p_0=\frac{16-k^2}{32\kappa}\bigg\}.
\end{equation}
Finally, $\beta_{tip}$ and $\beta_{1}$ coincide on the part corresponding to $\gamma\in[1/\kappa,+\infty)$ of the following branch of a quartic, so--called \textcolor{blue}{blue} quartic $\mathcal{Q}$
\begin{equation}
\begin{split}
	p&=p_\mathcal{Q}(\gamma):=\frac{\kappa}{16}\bigg(1+\frac{\kappa}{4}\bigg)\gamma-\frac{\kappa}{2}\gamma^2-\frac{1}{8}\Delta^{\frac{1}{2}}(\gamma),\\
q&=q_\mathcal{Q}(\gamma):=p_\mathcal{Q}(\gamma)+\gamma-\kappa\gamma^2, \gamma\in\R
\end{split}
\end{equation}
where
\begin{equation}
	\Delta(\gamma)=4k^2\gamma^2-2\kappa(4+\kappa)\gamma+\frac{1}{4}(8+\kappa)^2+4\kappa.
\end{equation}
The intersection of the blue quartic $\mathcal{Q}$ with the green parabola $\mathcal{G}$ is located at
\begin{equation}
	Q_0: p_0'=-1-\frac{3\kappa}{8}, q_0'=-2-\frac{7\kappa}{8}.
\end{equation}
Note that the vertical line $D_0'$ also passes through $Q_0$. In \cite{DHLZ2018}, the authors conjectured that $\beta(p,q)$ is respectively given by $\beta_{tip},\beta_{0},\beta_{lin},\beta_1$ on the domains $\mathfrak{D}_{tip}, \mathfrak{D}_{0}, \mathfrak{D}_{lin}, \mathfrak{D}_{1}$ defined as: $\mathfrak{D}_{tip}$ is the domain located above the blue quartic $\mathcal{Q}$ up to point $Q_0$ and to the left of the upper half-line $D'_0$ starting at $Q_0$; $\mathfrak{D}_0$ is the domain located to the right of the upper half-line $D'_0$ starting from $Q_0$, above the part $Q_0P_0$ of the left branch of the green parabola $\mathcal{G}$, to the left of the upper half-line $D_0$ starting from $P_0$; $\mathfrak{D}_{lin}$ is the infinite wedge of apex $P_0$ situated to the right of the line $D_0$ and above the line $D_1$; $\mathfrak{D}_1$ is the domain situated below the line $D_1$, to the right of the left branch of the green parabola $\mathcal{G}$ above the point $Q_0$ and to the right of the blue quartic $\mathcal{Q}$ below the point $Q_0$ (see Figure \ref{fig_SLE_generalized_spectrum_diagram}). In fact, this is the only possible scenario that emerges to construct the average generalized integral means spectrum by a continuous matching of the four spectra $\beta_{tip},\beta_{0},\beta_{lin},\beta_1$ along the phase transition lines described above.

In the same work, the authors developed the methods in \cite{BS2009,DNNZ2015,BDZ2017} to the generalized spectrum and proved this conjecture for a large domain of the $(p,q)$--plane that includes the domains corresponding to the phases $\beta_{tip},\beta_{0},\beta_{lin}$ and a part of the domain corresponding to the phase $\beta_1$. Namely, they proved the validity of the conjecture to the domain situated to the left of the right branch of the green parabola $\mathcal{G}$ up to the intersection point with the line $D_3$: $q-p=-1-\kappa/2$, followed by the line $D_3$ up to the intersection point of this line with the right part of red parabola $\mathcal{R}$ starting from $P_0$, followed by the red parabola up to point $P_0$, then followed by the upper half-line $D_1$: $q=p-(k^2-16)/32\kappa$ (see Figure \ref{fig_old_validity_zone}). They also showed that $\beta(p,q)$ is bounded above by $\beta_1(p,q)$ on the rest of of the $(p,q)$--plane. The corresponding results on the average generalized integral means spectrum $\beta^{[m]}$ associated to the m--fold transform $f^{[m]}$ of the whole--plane SLE map was also obtained by the relation
\begin{equation}\label{eq_relation_gener_spec_m fold and standard_intro}
	\begin{split}
	&\beta^{[m]}(p,q)=\beta(p,q_m),\\
	&q_m=q_m(p,q)=(1-1/m)p+q/m.
	\end{split}
\end{equation}

 The main result of this paper is an extension of the above validity zone of the conjecture about the average generalized integral means spectrum of SLE. Let us denote $\mathcal{I}$ the interior domain whose boundary is the \textcolor{red}{red} parabola $\mathcal{R}$ (see Figure \ref{fig_E_I_E}).
 \begin{thm}\label{thm_main_theorem}
 	The average generalized integral means spectrum $\beta(p,q)$ of whole--plane $\SLE$ in the domain $\mathfrak{D}_1\cap \mathcal{I}$ is given by
 	\begin{equation}
 	\beta_1(p,q)=3p-2q-\frac{1}{2}-\frac{1}{2}\sqrt{1+2\kappa(p-q)}.
 	\end{equation}
 \end{thm}
The validity zone of the conjecture is thus extended up to the part of the \textcolor{red}{red} parabola located inside the domain $\mathfrak{D}_1$ (see Figure \ref{fig_new_validity_zone}).
\begin{figure}[httb]
	\centering
	\begin{subfigure}{0.5\textwidth}
		\centering
		\includegraphics[width=0.9\linewidth]{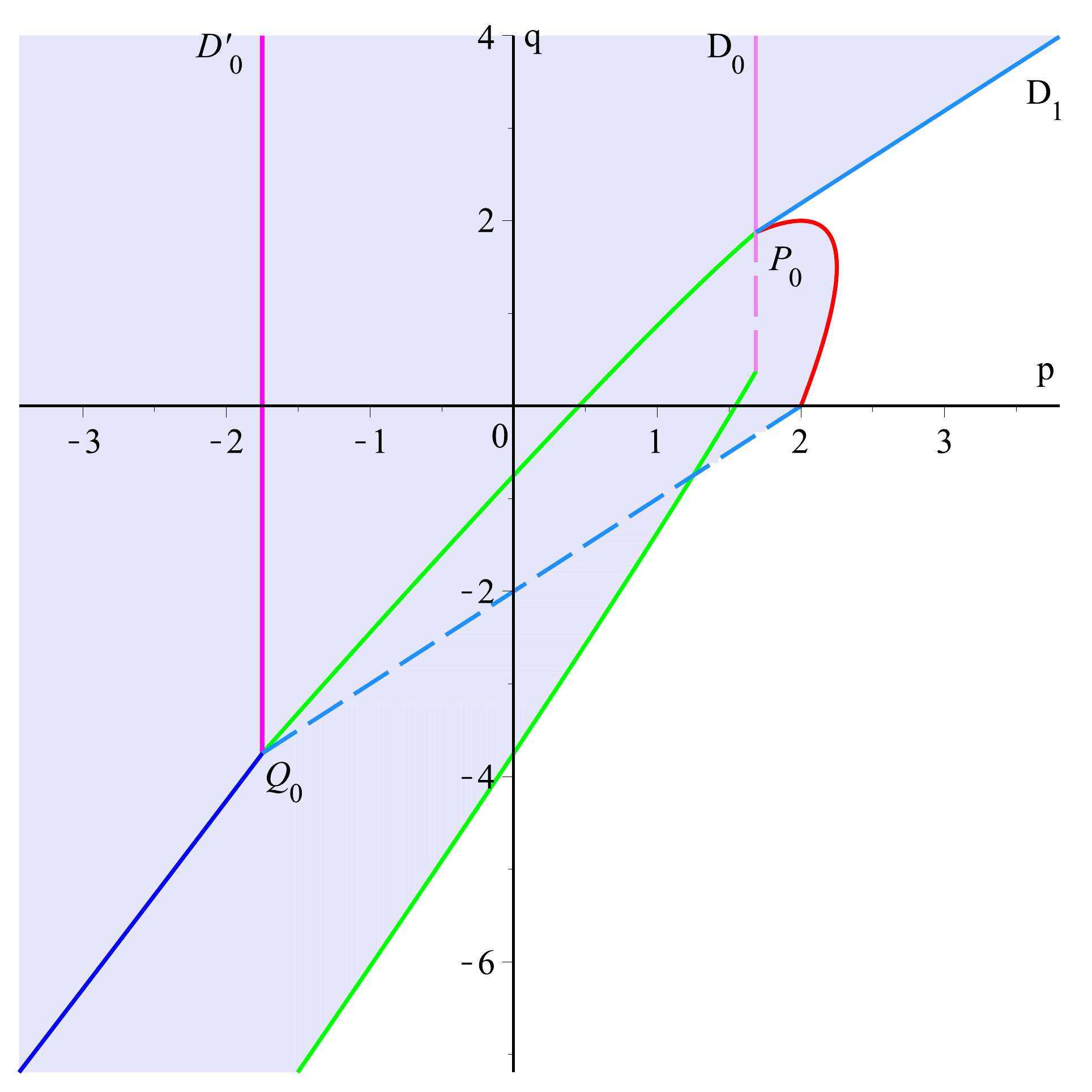}
		\caption{Zone of validity before this work}
		\label{fig_old_validity_zone}
	\end{subfigure}%
	\begin{subfigure}{0.5\textwidth}
		\centering
		\includegraphics[width=0.9\linewidth]{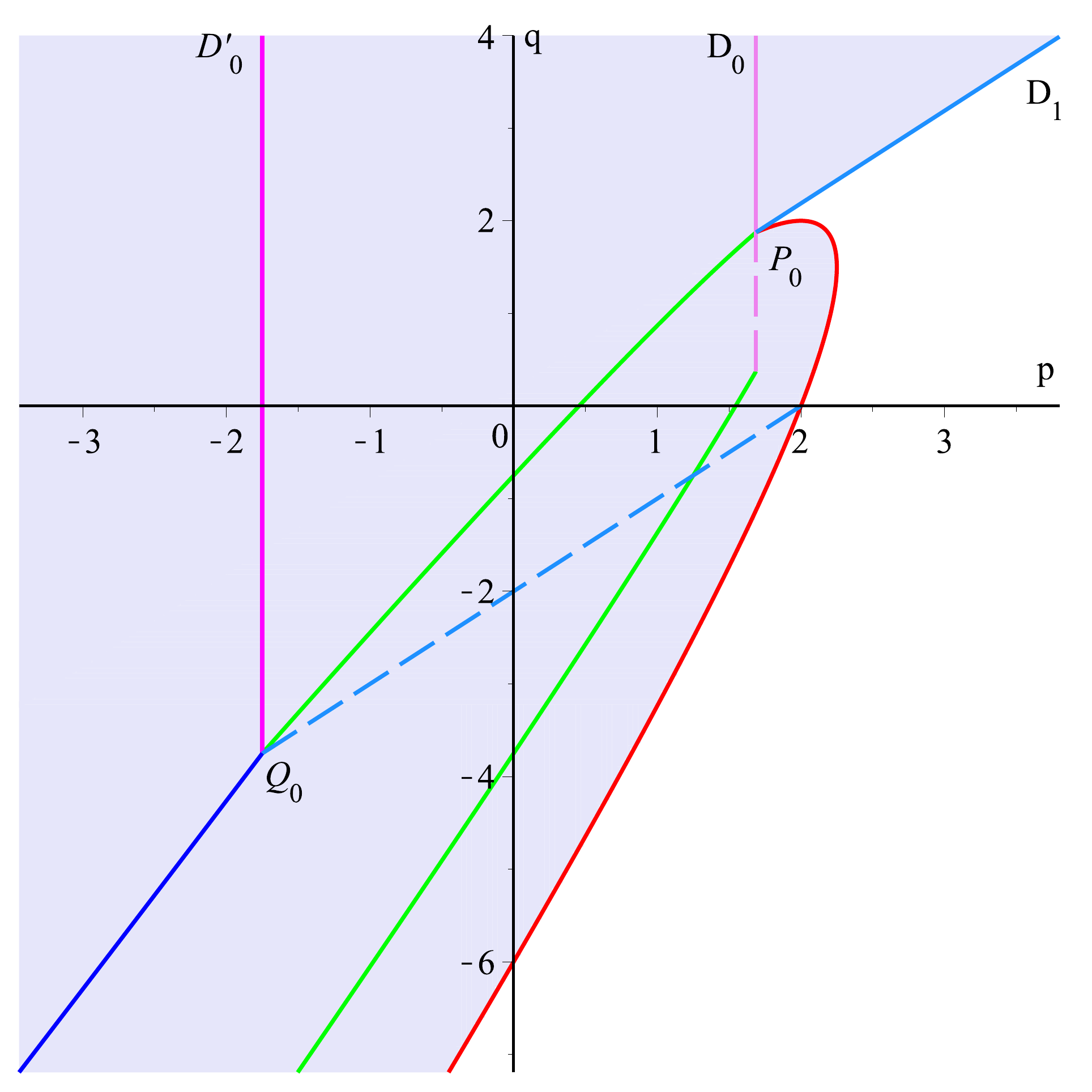}
		\caption{New zone of validity}
		\label{fig_new_validity_zone}
	\end{subfigure}
	\caption{{\it Validity zones  of the conjecture about $\beta(p,q)$: The figures illustrate the case $\kappa=2$.}}
	\label{whpl2}
\end{figure}

 The following corollary extends the domain of parameters for which the usual integral means spectrum of the interior whole-plane SLE is rigorously computed
 \begin{cor}
 	For $p\in [p^{*}(\kappa),p(\kappa)]$, where $p^{*}(\kappa)$ and $p(\kappa)$ are defined by \eqref{p*}, \eqref{p intersec red-ox}, the average integral means spectrum \eqref{average int means spec} of whole-plane SLE is given by
 	\begin{equation}
 	\beta_1(p,0)=3p-\frac{1}{2}-\frac{1}{2}\sqrt{1+2\kappa p}.
 	\end{equation}
 \end{cor}

Theorem \ref{thm_main_theorem} also improves the results obtained in \cite{DHLZ2018} on the generalized integral means spectrum associated to the m--fold transform of the whole--plane SLE map (see Figure \ref{fig_m-fold_new_validity_zone}). Let $\mathfrak{D}_1^{[m]}$ and $\mathcal{I}^{[m]}$ be the images of the domains $\mathfrak{D}_1$, $\mathcal{I}$ under the inverse transformation $T_m^{-1}(p,q)=(p,(1-m)p+mq)$ of the endomorphism of $\R^2$ given by $T_m(p,q)=(p,q_m)$, where $q_m$ is defined in \eqref{eq_relation_gener_spec_m fold and standard_intro}.
\begin{cor}
		In the domain $\mathfrak{D}_1^{[m]}\cap\mathcal{I}^{[m]}$, the average generalized integral means spectrum $\beta^{[m]}(p,q)$ associated to the m--fold transform of the whole--plane SLE map $f=f_0$ is given by
	\begin{equation}
	\beta_1^{[m]}(p,q)=\beta_1(p,q_m)=\bigg(1+\frac{2}{m}\bigg)p-\frac{2}{m}q-\frac{1}{2}-\frac{1}{2}\sqrt{1+\frac{2\kappa}{m}(p-q)}.
	\end{equation}
\end{cor}
\begin{figure}[httb]
	\centering
	\begin{subfigure}{0.5\textwidth}
		\centering
		\includegraphics[width=0.9\linewidth]{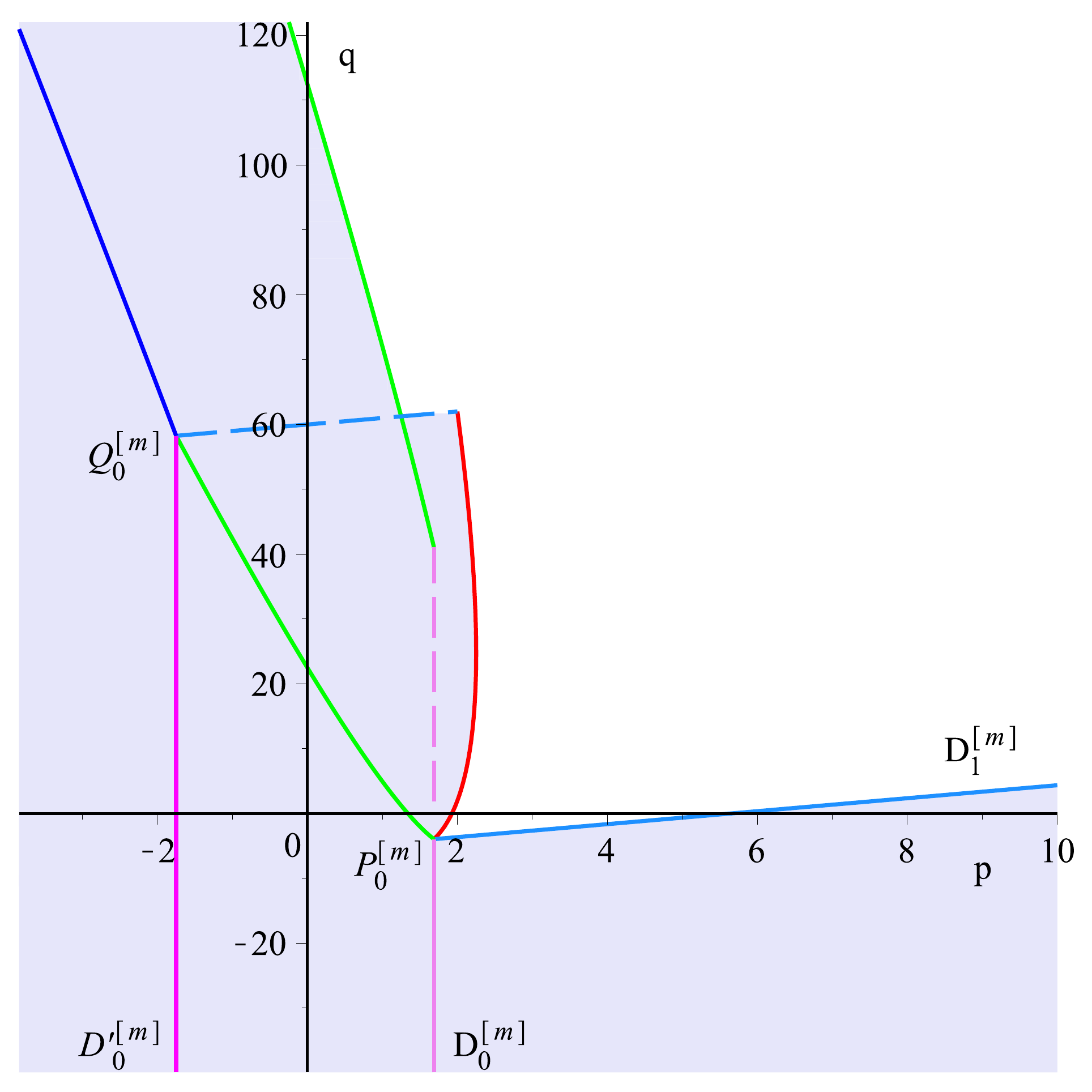}
		\caption{Zone of validity before this work}
		\label{fig_m-fold_old_validity_zone}
	\end{subfigure}%
	\begin{subfigure}{0.5\textwidth}
		\centering
		\includegraphics[width=0.9\linewidth]{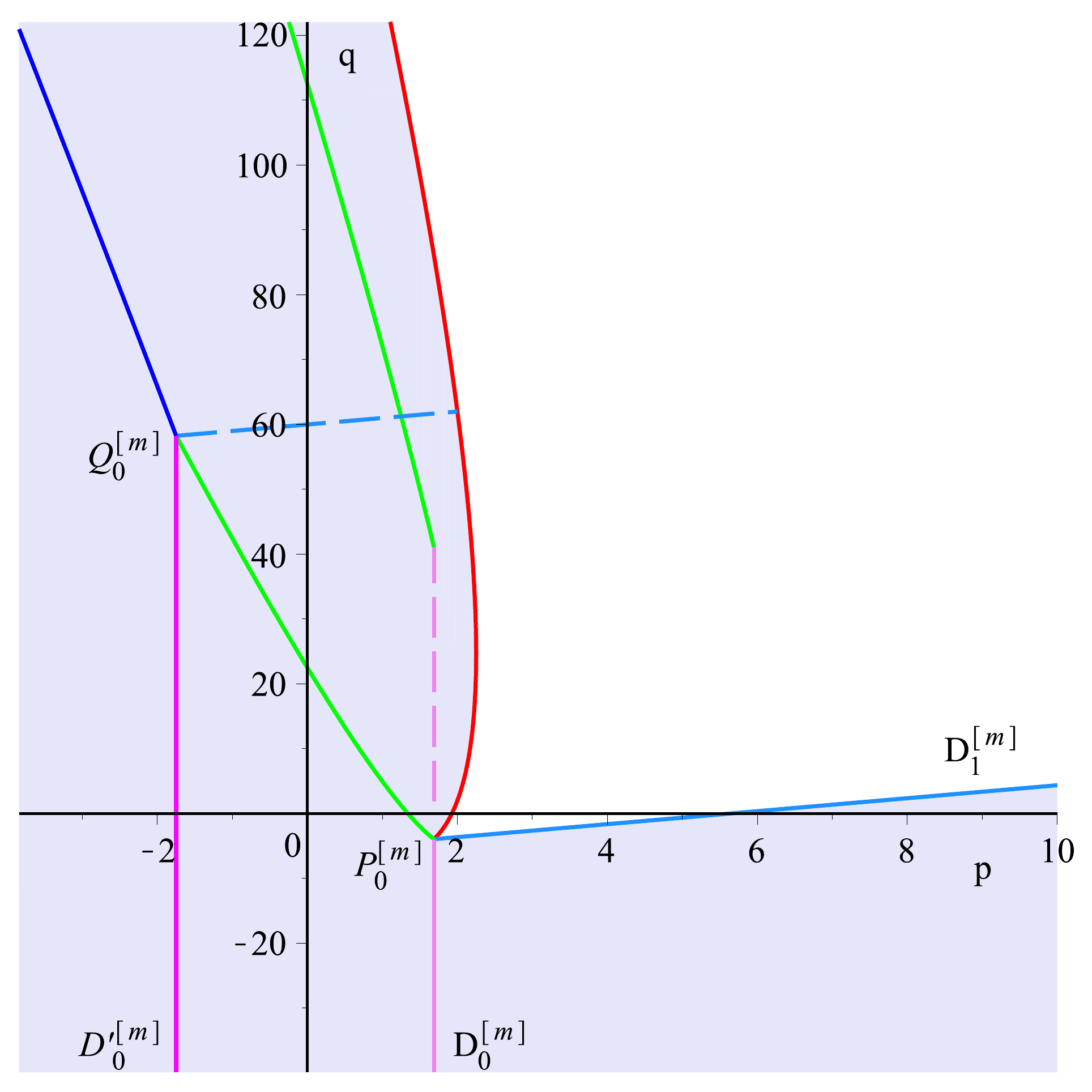}
		\caption{New zone of validity}
		\label{fig_m-fold_new_validity_zone}
	\end{subfigure}
	\caption{{\it Validity zones  of the conjecture about the spectrum $\beta^{[m]}(p,q)$: The figures illustrate the case $\kappa=2, m=-30$.}}
	\label{whpl2}
\end{figure}

\subsection{Schramm--Loewner Evolution $\SLE_\kappa$} \label{section intro SLE}
\begin{defn}
	The {\it Whole--Plane Schramm--Loewner evolution} (of parameter $\kappa$), denoted by $\SLE_\kappa$, in the unit disc is the solution of the stochastic PDE:
	\begin{equation}\label{loewner whole-plane-interior}
	\frac{\partial f_t(z)}{\partial t}=z\frac{\partial f_t(z)}{\partial z}\frac{\lambda(t)+z}{\lambda(t)-z},\hspace{0.5cm} \lim_{t\to +\infty}f_t(e^{-t}z)=z.
	\end{equation}
	where $\lambda(t):=e^{i\sqrt{\kappa}B_t}$ with $B_t$ being the standard one dimensional Brownian motion and $\kappa$ being a nonnegative parameter.
\end{defn}
There is another variant of SLE, called {\it{radial Schramm--Loewner evolution}}.
\begin{defn}
	The {\it Radial Schramm--Loewner evolution} (of parameter $\kappa$), in the unit disc is the solution of the stochastic PDE:
	\begin{equation}\label{eq interior whole-plane SLE f}
	\frac{\partial f_t(z)}{\partial t}=-z\frac{\partial f_t(z)}{\partial z}\frac{\lambda(t)+z}{\lambda(t)-z}, \hspace{0.3cm} f_0(z)=z,
	\end{equation}
\end{defn}
It has been proved that the radial SLE map $f_t$ has the same law as the solution $\tilde f_t$  of the following stochastic ODE \cite{DNNZ2015,DHLZ2018}
\begin{equation}\label{reverseeq}
\partial_t\tilde f_t(z)=\tilde f_t(z)\frac{\tilde f_t(z)+\lambda(t)}{\tilde f_t(z)-\lambda(t)},\,\,\,\tilde f_0(z)=z.
\end{equation}
We have the two following properties of radial SLE map $\tilde{f}_t$: The first one is a kind of Markov property and the second one is a relation between the radial SLE and the whole--plane SLE \cite{DNNZ2015,DHLZ2018}
\begin{lem}(Markov property of radial SLE map)
	\begin{equation}\label{markov-radialSLE}
	\tilde{f}_t(z)=\lambda(s)\tilde{f}_{t-s}(\tilde{f}_{s}(z)/\lambda(s)).
	\end{equation}
\end{lem}
\begin{lem}\label{lemmefond}
	The limit in law, $\lim_{t\to +\infty} e^{t}\tilde{f}_t(z)$,  exists and has the {\it same law} as the (time 0) interior whole--plane random map $f_0(z)$:
	\begin{equation}
		\lim_{t\to +\infty} e^t \tilde f_t(z)\stackrel{\rm (law)}{=} f_0(z).
	\end{equation}
\end{lem}

Note that the Loewner chains can also be defined on the complement of the unit disc. Let us here give the definition of the {\it exterior} version of whole--plane SLE
\begin{defn}
	The {\it Whole--Plane Schramm--Loewner evolution} (of parameter $\kappa$), denoted by $\SLE_\kappa$, in the complement of the unit disc is the solution of the stochastic PDE:
	\begin{equation}\label{loewner whole-plane-exterior}
	\frac{\partial \hat{f}_t(\zeta)}{\partial t}=\zeta\frac{\partial \hat{f}_t(\zeta)}{\partial \zeta}\frac{\lambda(t)+\zeta}{\lambda(t)-\zeta}, \hspace{0.5cm} \lim_{t\to +\infty}\hat{f}_t(e^{t}z)=z,
	\end{equation}
	where $\lambda(t):=e^{i\sqrt{\kappa}B_t}$ with $B_t$ being the standard one dimensional Brownian motion and $\kappa$ being a nonnegative parameter.
\end{defn}
The exterior whole--plane $\SLE_\kappa$ map $\hat{f}_t$ and the interior whole--plane $\SLE_\kappa$ map $f_t$ are related by the inversion
\begin{equation}\label{eq inter-exter SLE relation}
	\hat{f}_t(\zeta)\stackrel{\rm (law)}{=}\frac{1}{f_t(\frac{1}{\zeta})}, \hspace{0.3cm}\zeta\in\Delta.
\end{equation}

\subsection{Average generalized integral means spectrum of whole--plane SLE}
It is showed in \cite{DHLZ2015arXiv,DHLZ2018} that by using a martingale method and Markov property \eqref{markov-radialSLE}, one can obtain a PDE satisfied by the mixed moment $\mathbb{E}(\vert z\vert^q\vert \tilde{f}'_t(z) \vert^p/\vert \tilde{f}_t(z) \vert^q)$ of the radial SLE map $\tilde{f}_t$. Lemma \ref{lemmefond} then follows to show that function $G(z,\bar{z})= \mathbb{E}(\vert z \vert ^q\vert f'(z) \vert^p/\vert f(z) \vert^q)$ is a solution of the following linear PDE
 \begin{align}\label{eq1zz}  \mathcal{P}(D)[G(z,\bar{z})]&=-\frac{\kappa}{2}(z\partial-\bar{z}\bar{\partial})^2G-\frac{1+z}{1-z}z\partial G-\frac{1+\bar{z}}{1-\bar{z}}\bar{z}\bar{\partial} G\\ \nonumber
 +\bigg[ &-\frac{p}{(1-z)^2}-\frac{p}{(1-\bar{z})^2}+\frac{q}{1-z}+\frac{q}{1-\bar{z}}+2p-2q \bigg]G=0.
 \end{align}
The solution space of this equation is one--dimensional and $G$ is the solution satisfying $G(0,0)=1$. Using polar coordinates ($z=re^{i\theta}$), we can rewrite this equation as
\begin{align}
	&\frac{\kappa}{2}G_{\theta,\theta}-\frac{2r\sin\theta}{r^2-2r\cos\theta+1}G_\theta+\frac{r(r^2-1)}{r^2-2r\cos\theta+1}G_r \label{eq_main_polar}\\
	+\bigg( -2p&\frac{1-2r\cos\theta +r^2(\cos^2\theta-\sin^2\theta)}{(1-2r\cos\theta+r^2)^2} +2q\frac{1-r\cos\theta}{1-2r\cos\theta+r^2}+2p-2q\bigg)G=0.\nonumber
\end{align}
Eq. \eqref{eq_main_polar} is parabolic and the variable $r$ plays the role of time in the heat equation. The singularity of \eqref{eq_main_polar} at $r=1$ corresponds to the singularity at infinity in the usual parabolic equations. Moreover, the coefficients of \eqref{eq_main_polar} have singularities at $(r,\theta)=(1,0)$. Coming back to the variables $z,\bar{z}$, we just look for the singular part of $G(z,\bar{z})$ as
\begin{equation}\label{singular part z, z bar}
(1-z\bar{z})^{-\beta}u^\gamma,\,\, \text{ where } u=(1-z)(1-\bar{z}).
\end{equation}
The average generalized integral means spectrum $\beta(p,q)$ is then given by
\begin{equation}
	\beta(p,q)=\left\{
	\begin{array}{ll}
	\beta & \gamma\geq-1/2 \\
	\beta-2\gamma-1 & \gamma<-1/2.
	\end{array} 
	\right. 
\end{equation} 

{\bf Spectrum functions.} It is so far showed in \cite{BDZ2016,BDZ2017,BS2009,DHLZ2015arXiv,DHLZ2018,DNNZ2011,DNNZ2015} that the values of $\beta$ and $\gamma$ are determined by particular quadratic functions. Let us now recall how these functions appeared. We use the same notations as in \cite{DHLZ2015arXiv,DHLZ2018}.
We focus on the solutions $G(z,\bar{z})$ of \eqref{eq1zz} of the form \eqref{singular part z, z bar}
\begin{equation}
	G(z,\bar{z}):=(1-z\bar{z})^{-\beta}u^\gamma,\,\, \text{ where } u=(1-z)(1-\bar{z}).
\end{equation}
The action of the differential operator $\mathcal{P}(D)$ on $G(z,\bar{z})$ is then
\begin{align}
	\mathcal{P}(D)(G(z,\bar{z}))=&2\bigg[\beta(p,\gamma)-\beta\bigg]\frac{z\bar{z}}{u}+C(p,\gamma)\bigg(\frac{1-z\bar{z}}{u}\bigg)^2\\
	&+[A(p,q,\gamma)+C(p,\gamma)]\frac{1-z\bar{z}}{u}+A(p,q,\gamma),\nonumber
\end{align}
 where
\begin{align}
	&A(p,q,\gamma):=-\frac{\kappa}{2}\gamma^2+\gamma+p-q,\\
	&C(p,\gamma):=-\frac{\kappa}{2}\gamma^2+\bigg(2+\frac{\kappa}{2}\bigg)\gamma- p,\\
	&\beta(p,\gamma):=\frac{\kappa}{2}\gamma^2-C(p,\gamma)=k\gamma^2-\bigg(2+\frac{\kappa}{2}\bigg)\gamma+ p.\label{eq_beta}
\end{align}
We will often use the shorthand notations $A^{\sigma}(\gamma),C(\gamma),\beta(\gamma)$ for $A(p,q,\gamma), C(p,\gamma), \beta(p,\gamma)$, where $\sigma=q/p-1$. It is obvious that on the curve in the $(p,q)-$plane defined by
\begin{equation}
\begin{split}
A^{\sigma}(\gamma)&=-\frac{\kappa}{2}\gamma^2+\gamma+p-q=0,\\
C(\gamma)&=-\frac{\kappa}{2}\gamma^2+\bigg(2+\frac{\kappa}{2}\bigg)\gamma- p=0,
\end{split}
\end{equation}
 where $\gamma$ is a real parameter, we have a solution $G$ of \eqref{eq1zz} as
\begin{equation}\label{eq G on red}
	G(z,\bar{z})=(1-z\bar{z})^{-\beta(\gamma)}u^\gamma=(1-z\bar{z})^{-\frac{k\gamma^2}{2}}(1-z)^\gamma(1-\bar{z})^\gamma.
\end{equation}
On this curve, the average generalized integral means spectrum is thus given in term of the parameter $\gamma$ by
\begin{equation}
\beta(p,q)=\left\{
\begin{array}{ll}
\beta(\gamma) & \gamma\geq-1/2 \\
\beta(\gamma)-2\gamma-1 & \gamma<-1/2.
\end{array} 
\right. 
\end{equation}
We see that the function $\beta(p,q)$ is determined by the quadratic functions $A^{\sigma}(\gamma),C(\gamma),\beta(\gamma)$. As we will see in the rest of this introduction, all analytic expressions so far obtained  of the spectrum $\beta(p,q)$ are defined by means of these quadratic functions.\\

{\it Dual property of the function $\beta(\gamma)$.}
Function $\beta$ defined by \eqref{eq_beta} has an important {\it dual property}: let $\gamma$ and $\gamma'$ such that
\begin{equation}\label{eq def dual points}
\gamma+\gamma'=\frac{2}{\kappa}+\frac{1}{2},
\end{equation}
then $\beta(\gamma)=\beta(\gamma')$. Denote $\gamma_{lin}$ the value of the parameter $\gamma$ at which $\gamma$ and $\gamma'$ coincide
\begin{equation}
	\gamma_{lin}:=\frac{1}{\kappa}+\frac{1}{4}.
\end{equation}
It is remarkable that $\beta(\gamma_{lin})= \min_\gamma \beta(\gamma)$.\\

{\it Spectrum functions.} Let us denote $\gamma_1^-, \gamma_1$ the real solutions of $A^{\sigma}(\gamma)=0$ and $\gamma_0, \gamma_0^+$ the real solutions of $C(\gamma)=0$:
\begin{align}
	\gamma_0&=\frac{2}{\kappa}+\frac{1}{2}-\frac{1}{2\kappa}\sqrt{(4+\kappa)^2-8\kappa p},\\
	\gamma_0^+&=\frac{2}{\kappa}+\frac{1}{2}+\frac{1}{2\kappa}\sqrt{(4+\kappa)^2-8\kappa p},\\
	\gamma_1^-&=\frac{1}{\kappa}-\frac{1}{\kappa}\sqrt{1-2\sigma\kappa p}=\frac{1}{\kappa}-\frac{1}{\kappa}\sqrt{1+2\kappa(p-q)},\\
	\gamma_1&=\frac{1}{\kappa}+\frac{1}{\kappa}\sqrt{1-2\sigma\kappa p}=\frac{1}{\kappa}+\frac{1}{\kappa}\sqrt{1+2\kappa(p-q)}.\label{eq def gamma_1}
\end{align}
\begin{rem}
	 The solutions $\gamma_0, \gamma_0^+$ are defined only to the left of the vertical line
	 \begin{equation}\label{eq line Delta 0}
	 	\Delta_0: p=-1-3\kappa/8.
	 \end{equation}
	While $\gamma_1^-,\gamma_1$ are defined only below the line of slope 1
	\begin{equation}\label{eq line Delta 1}
		\Delta_1: q=p+1/2\kappa.
	\end{equation}
\end{rem}

\begin{defn}
	The {\it spectrum functions} $\beta_{0}, \beta_{tip}, \beta_{1}, \beta_{lin}$ are defined as
	\begin{align}
	\beta_{0}(p)&:=\beta(\gamma_0)=-p+\frac{4+\kappa}{4\kappa}(4+\kappa-\sqrt{4+\kappa)^2-8\kappa p}),\\
	\beta_{tip}(p)&:=\beta(\gamma_0)-2\gamma_0-1=-p-1+\frac{1}{4}(4+\kappa-\sqrt{4+\kappa)^2-8\kappa p}),\label{eq def beta tip}\\
	\beta_{1}(p,q)&:=\beta(\gamma_1)=3p-2q-\frac{1}{2}-\frac{1}{2}\sqrt{1+2\kappa(p-q)},\\
	\beta_{lin}(p)&:=\beta(\gamma_{lin})=p-\frac{(4+\kappa)^2}{16\kappa}.
	\end{align}
\end{defn}

We call $\mathcal{S}$ the infinite wedge located below the line $\Delta_1$ \eqref{eq line Delta 1} and to the left of the line $\Delta_0$ \eqref{eq line Delta 0}. This is the domain where both $\gamma_0, \gamma_1$, hence $\beta_{0}$ and $\beta_{1}$, are well defined.\\

Return to values of the average generalized integral means spectrum $\beta(p,q)$, on the above curve $A^{\sigma}(\gamma)=0, C(\gamma)=0$, the value of $\beta(p,q)$ is given by
\begin{align}
\beta(p,q)&=\left\{
\begin{array}{ll}
\beta(\gamma_0)-2\gamma_0-1& \gamma=\gamma_0\leq-1/2 \\
\beta(\gamma_0) & -1/2<\gamma=\gamma_0\leq\gamma_{lin}\\
\beta(\gamma_1) & \gamma=\gamma_1>\gamma_{lin}\\
\end{array} 
\right.\\
&=\left\{
\begin{array}{ll}
\beta_{tip}(p)& p\leq p'_0(\kappa)\\
\beta_0(p) & p'_0(\kappa)<p\leq p_0(\kappa)\\
\beta_1(p,q) & p>p_0(\kappa)\\
\end{array} 
\right.
\end{align}
where $p'_0(\kappa)=-1-3\kappa/8$ and $p_0(\kappa)=3(4+\kappa)^2/32\kappa$.

In \cite{BDZ2016,BDZ2017,BS2009} the average integral means spectrum of the exterior whole--plane SLE, corresponding to the spectrum $\beta(p,q)$ on the line $q=2p$, was obtained as
\begin{equation}
	\beta(p,2p)=\left\{
	\begin{array}{ll}
	\beta_{tip}(p)& p\leq p'_0(\kappa)\\
	\beta_0(p) & p'_0(\kappa)<p\leq p_0(\kappa)\\
	\beta_{lin}(p) & p>p_0(\kappa)\\
	\end{array} 
	\right.
\end{equation}

In \cite{DNNZ2011,DNNZ2015} the average integral means spectrum of the interior whole--plane SLE, corresponding to the spectrum $\beta(p,q)$ on the line $q=0$, was obtained as
\begin{equation}
\beta(p,0)=\left\{
\begin{array}{ll}
\beta_{tip}(p)& p\leq p'_0(\kappa)\\
\beta_0(p) & p'_0(\kappa)<p\leq p^*(\kappa)\\
\beta_{1}(p,0) & p^*(\kappa)<p<\min(\hat{p}(\kappa),p(\kappa))\\
\end{array} 
\right.
\end{equation}
where
\begin{align}
	p^*(\kappa)&=(\sqrt{2(4+\kappa)^2+4}-6)(\sqrt{2(4+\kappa)^2+4}+2)/32\kappa,\label{p*}\\
	 \hat{p}(\kappa)&=1+\kappa/2,\label{p hat}\\
	 p(\kappa)&=(6+\kappa)(2+\kappa)/8\kappa.\label{p intersec red-ox}
\end{align}

{\bf Transition lines.} We recall the subsets of the parameter $(p,q)$--plane where the spectrum functions $\beta_{tip},\beta_{0},\beta_{lin},\beta_{1}$ coincide.\\

{\it The line $D'_0$.} Since the definition \eqref{eq def beta tip} of $\beta_{tip}$, the spectrum functions $\beta_{tip}$ and $\beta_{0}$ coincide on the set characterized by $\gamma_0(p)=-1/2$, that is the vertical line
\begin{equation}
	D'_0:=\bigg\{(p,q): p=-1-\frac{3\kappa}{8}\bigg\}.
\end{equation}

{\it The Red parabola and the Green parabola.} To find the set where $\beta_{0}$ and $\beta_{1}$ coincide, we look for the solution of the equation $\beta(\gamma_0(p))=\beta(\gamma_1(p,q))$. This equation is equivalent to $\gamma_0(p)=\gamma_1(p,q)$ or $\gamma_0(p)=\gamma_1'(p,q)$ where $\gamma_1'(p,q)$ is the dual value of $\gamma_1(p,q)$ by the Dual property of $\beta$. The last equations lead to introducing the two parabolas so--called (and drawn in) \textcolor{red}{red}, denoted by $\mathcal{R}$ and \textcolor{green}{green}, denoted by $\mathcal{G}$.\\
\indent
The {\it red parabola $\mathcal{R}$} is the previously mentioned curve on which we find the integrable form \eqref{eq G on red} of the generalized moment $\mathbb{E}(\vert z \vert^q \vert f'(z)) \vert^p/\vert f(z)\vert^q)$. Recall that this parabola is defined by the simultaneous conditions
\begin{equation}
	A^\sigma(p,\gamma)=A(p,q,\gamma)=0, C(p,\gamma)=0.
\end{equation}
Its parametric form is given by
\begin{equation}
\begin{split}
p&=p_\mathcal{R}(\gamma):=\bigg(2+\frac{\kappa}{2}\bigg)\gamma-\frac{\kappa}{2}\gamma^2,\\
q&=q_\mathcal{R}(\gamma):=\bigg(3+\frac{\kappa}{2}\bigg)\gamma-\kappa\gamma^2, \gamma\in\R.
\end{split} 
\end{equation}
Along the red parabola $\mathcal{R}$, we successively have
\begin{align}
	\gamma&=\gamma_1^{-}(p)=\gamma_0 (p), \gamma\in(-\infty,1/\kappa],\\
	\gamma&=\gamma_1(p)=\gamma_0 (p), \gamma\in[1/\kappa,2/\kappa+1/2],\label{eq gamma red middle}\\
	\gamma&=\gamma_1(p)=\gamma_0^+ (p), \gamma\in[2/\kappa+1/2,+\infty).
\end{align}
The above three intervals of the parameter $\gamma$ are respectively corresponding to the three parts of the red parabola $\mathcal{R}$ that we hereafter call the {\it left part}, the {\it middle part} and the {\it right part} of $\mathcal{R}$. These three parts are separated by the tangent points $T_1, T_0$ of the red parabola with the tangent lines $\Delta_1, \Delta_0$ respectively defined by \eqref{eq line Delta 1}, \eqref{eq line Delta 0} (see Figure \ref{fig_red_green}). From the definitions of $\beta_{0},\beta_{1}$ and \eqref{eq gamma red middle}, we see that $\beta_{0}$ and $\beta_{1}$ coincide on the middle part of the red parabola $\mathcal{R}$.\\
\indent
The {\it green parabola $\mathcal{G}$} is defined by the simultaneous conditions
\begin{equation}
\begin{split}
&A^{\sigma}(p,\gamma')=A(p,q,\gamma)=0,\,\,C(p,\gamma)=0,\\
&\gamma+\gamma'=2/\kappa+1/2.
\end{split}
\end{equation}
Here we have set that $\gamma$ and $\gamma'$ are {\it dual} of each other by the Dual property of the function $\beta$.
We find the parametric form of $\mathcal{G}$ as
\begin{equation}
\begin{split}
p&=p_\mathcal{G}(\gamma'):=\frac{(4+\kappa)^2}{8\kappa}-\frac{\kappa}{2}\gamma'^2,\\
q&=q_\mathcal{G}(\gamma'):=\frac{(4+\kappa)^2}{8\kappa}+\gamma'-\kappa\gamma'^2, \gamma'\in\R.
\end{split}
\end{equation}
Along the parabola $\mathcal{G}$, we have
\begin{align}
	\gamma'&=\gamma_1^{-}(p), \gamma=\gamma_0^+ (p),\gamma\in(-\infty,0],\\
	\gamma'&=\gamma_1^{-}(p),\gamma=\gamma_0 (p), \gamma\in[0,1/\kappa],\\
	\gamma'&=\gamma_1(p), \gamma=\gamma_0 (p), \gamma\in[1/\kappa,+\infty).\label{eq gamma green right}
\end{align}
The intervals $(-\infty,1/\kappa]$ and $[1/\kappa,+\infty)$ of the parameter $\gamma$ are respectively corresponding to the two parts of the green parabola $\mathcal{G}$ that we hereafter call the {\it right branch} and the {\it left branch} of $\mathcal{G}$. These two branches are separated by the tangent point $T'_1$ of the green parabola with the tangent line $\Delta_1$ (see Figure \ref{fig_red_green}). From the Dual property of the function $\beta$ and \eqref{eq gamma green right}, we see that $\beta_{0}$ and $\beta_{1}$ coincide on the left branch of the green parabola $\mathcal{G}$.\\

{\it The lines $D_0$ and $D_1$.} The intersection points of the red parabola and the green parabola are
\begin{align}
	P_0: p_0&=p_0(\kappa)=\frac{3(4+\kappa)^2}{32\kappa}, q_0=\frac{(4+\kappa)(8+\kappa)}{16\kappa},\\
	P_1: p_1&=\frac{(8+\kappa)(8+3\kappa)}{32\kappa},  q_0=\frac{(4+\kappa)(8+\kappa)}{16\kappa}.
\end{align}
Let $D_0$ and $D_1$ respectively be the vertical line and the slope 1 line passing through the point $P_0$
\begin{align}
	D_0:&=\bigg\{(p,q): p=p_0\bigg\}\\
	D_1:&=\bigg\{(p,q): q-p=q_0-p_0=\frac{16-\kappa^2}{32\kappa}\bigg\}.
\end{align}
The line $D_0$ is the set on which $\beta_{0}$ and $\beta_{lin}$ coincide while the line $D_1$ is the set on which $\beta_{1}$ and $\beta_{lin}$ coincide. Indeed, one can see that the equation $\beta(\gamma_0(p))=\beta(\gamma_{lin}(p))$ is equivalent to $\gamma_0(p)=\gamma_{lin}(p)$. The last one is then simplified to obtain $p=p_0$. Similarly, the equation $\beta(\gamma_1(p,q))=\beta(\gamma_{lin}(p))$ is equivalent to $\gamma_1(p,q)=\gamma_{lin}(p)$. The last one is then simplified to obtain $q-p=q_0-p_0$.\\

{\it The Blue Quartic.} To find the set where $\beta_{1}$ and $\beta_{tip}$ coincide, we look for the solution of the equation $\beta_1(p,q)=\beta_{tip}(p)$. It was showed in \cite{DHLZ2018} that the solution lies on a branch of a quartic $\mathcal{Q}$, so--called \textcolor{blue}{blue}, defined by
\begin{equation}
\begin{split}\label{eq blue quartic}
p&=p_\mathcal{Q}(\gamma):=\frac{\kappa}{16}\bigg(1+\frac{\kappa}{4}\bigg)\gamma-\frac{\kappa}{2}\gamma^2-\frac{1}{8}\Delta^{\frac{1}{2}}(\gamma),\\
q&=q_\mathcal{Q}(\gamma):=p_\mathcal{Q}(\gamma)+\gamma-\kappa\gamma^2, \gamma\in\R
\end{split}
\end{equation}
where
\begin{equation}
\Delta(\gamma)=4k^2\gamma^2-2\kappa(4+\kappa)\gamma+\frac{1}{4}(8+\kappa)^2+4\kappa.
\end{equation}
Along this branch of the quartic, we successively have
\begin{align}
	\gamma&=\gamma_1^{-}(p,q); \beta(\gamma_1^{-}(p,q))=\beta_{tip}(p), \gamma\in(-\infty,1/\kappa],\\
	\gamma&=\gamma_1(p,q); \beta(\gamma_1(p,q))=\beta_{tip}(p), \gamma\in[1/\kappa,+\infty).
\end{align}
Hence the spectrum functions $\beta_{1}$ and $\beta_{tip}$ coincide on the part of the quartic branch $\mathcal{Q}$ \eqref{eq blue quartic} corresponding to the interval $[1/\kappa,+\infty)$ of $\gamma$.
The intersection of the blue quartic $\mathcal{Q}$ \eqref{eq blue quartic} with the green parabola $\mathcal{G}$ is located at
\begin{equation}
Q_0: p_0'=-1-\frac{3\kappa}{8}, q_0'=-2-\frac{7\kappa}{8}; \gamma=1+\frac{2}{\kappa}.
\end{equation}
Note that the vertical line $D'_0$, corresponding to the transition from $\beta_{tip}$ to $\beta_{0}$, also passes through $Q_0$.\\
\begin{figure}[httb]
	\centering
	\begin{subfigure}{0.5\textwidth}
		\centering\captionsetup{width=.8\linewidth}
		\includegraphics[width=0.9\linewidth]{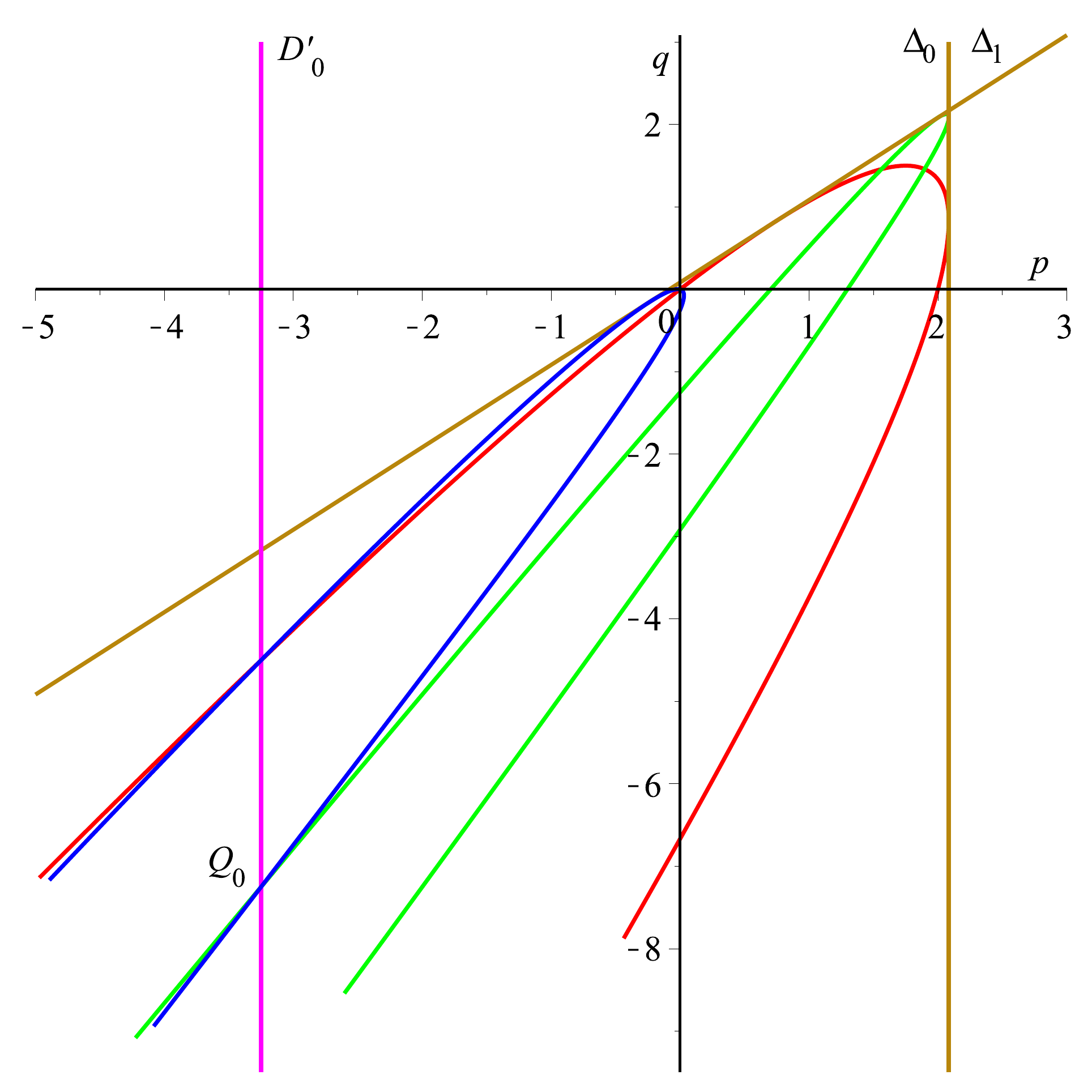}
		\caption{$D'_0$ passes through the intersection point $Q_0$ of the blue quartic $\mathcal{Q}$ and the left branch of the green parabola $\mathcal{G}$.}
		\label{fig transition blue green}
	\end{subfigure}%
	\begin{subfigure}{0.5\textwidth}
		\centering\captionsetup{width=.8\linewidth}
		\includegraphics[width=0.9\linewidth]{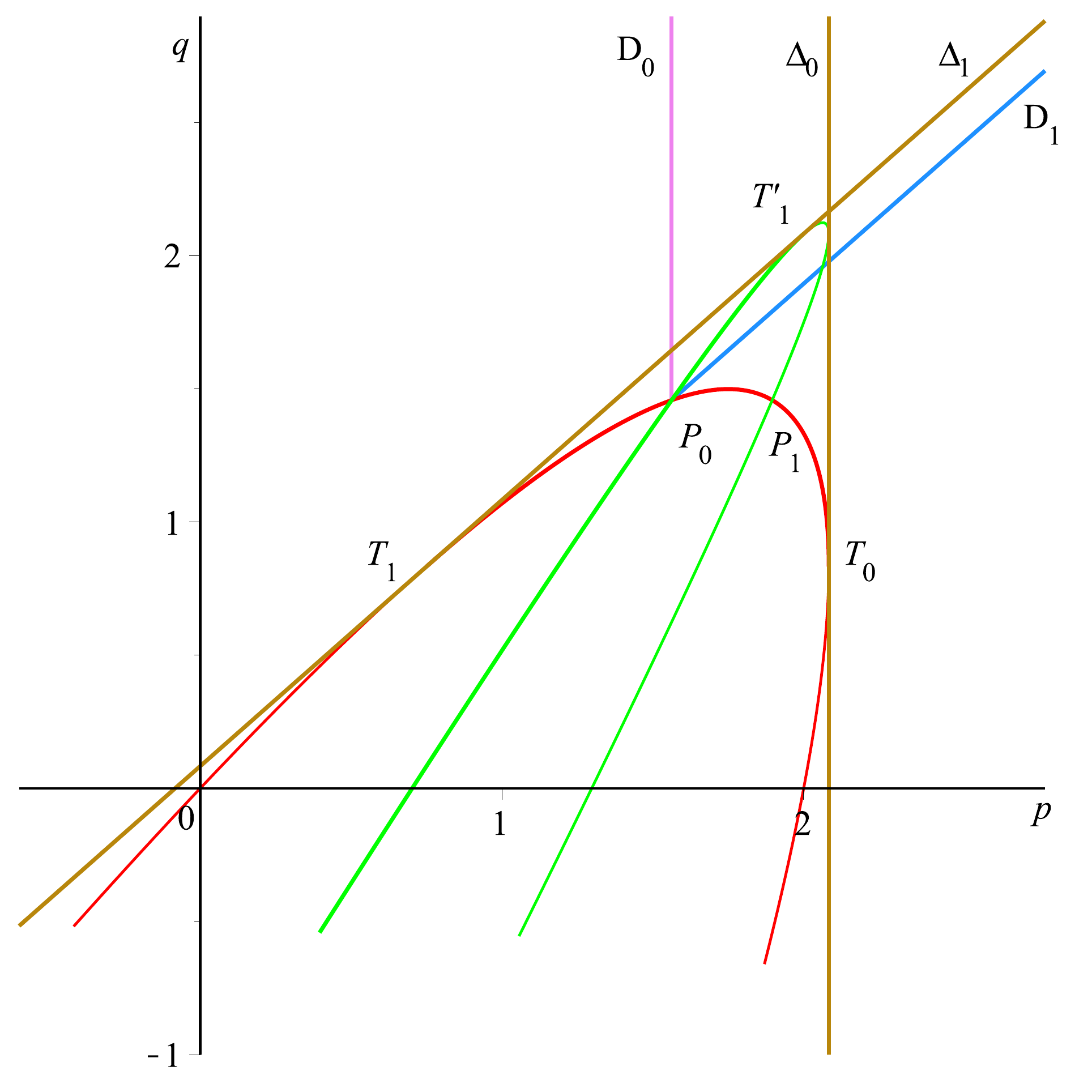}
		\caption{$P_0$ is a concurrency point of the middle part of the red parabola $\mathcal{R}$, the left branch of the green parabola $\mathcal{G}$, the lines $D_0$ and $D_1$.}
		\label{fig_red_green}
	\end{subfigure}
	\caption{{\it Various transition lines.}}
	\label{whpl2}
\end{figure}

{\bf Conjecture about the average generalized integral means spectrum of whole--plane SLE.} In \cite{DHLZ2015arXiv,DHLZ2018}, the authors claimed that the generalized integral means spectrum of whole--plane SLE has exactly four phases $\beta_{tip}, \beta_{0}, \beta_{lin}, \beta_{1}$. It comes from the belief (probably true) that the values of $\beta(p,q)$ are only given by the singularity of the moment expectation $\mathbb{E}(\vert z \vert ^q\vert f'(z) \vert^p/\vert f(z) \vert^q)$ at the generic part of the boundary of the image domain (the bulk), at the growing tip of the SLE trace, at the tip at infinity (unboundedness of $f=f_0$) or by the linear prolongation due to the convexity of the generalized spectrum. There is only one possible scenario that emerges to construct the average generalized integral means spectrum by a continuous matching of the four spectra $\beta_{tip},\beta_{0},\beta_{lin},\beta_1$ along the phase transition lines described above: $\beta(p,q)$ is respectively given by the functions $\beta_{tip}, \beta_{0}, \beta_{lin}, \beta_{1}$ in the four zones $\mathfrak{D}_{tip}, \mathfrak{D}_{0}, \mathfrak{D}_{lin}, \mathfrak{D}_{1}$ where these zones are defined as

$\mathfrak{D}_{tip}$ is the domain located above the blue quartic up to point $Q_0$ and to the left of the upper half-line $D'_0$ starting at $Q_0$.

$\mathfrak{D}_0$ is the domain located to the right of the upper half-line $D'_0$ starting from $Q_0$, above the part $Q_0P_0$ of the left branch of the green parabola $\mathcal{G}$, to the left of the upper half-line $D_0$ starting from $P_0$.

$\mathfrak{D}_{lin}$ is the infinite wedge of apex $P_0$ situated to the right of the line $D_0$ and above the line $D_1$.

$\mathfrak{D}_1$ is the domain situated below the line $D_1$, to the right of the left branch of the green parabola $\mathcal{G}$ above the point $Q_0$ and to the right of the blue quartic $\mathcal{Q}$ below the point $Q_0$  (see Figure \ref{fig_SLE_generalized_spectrum_diagram}).
\begin{figure}[h]
	\centering
	\includegraphics[width=0.5\linewidth]{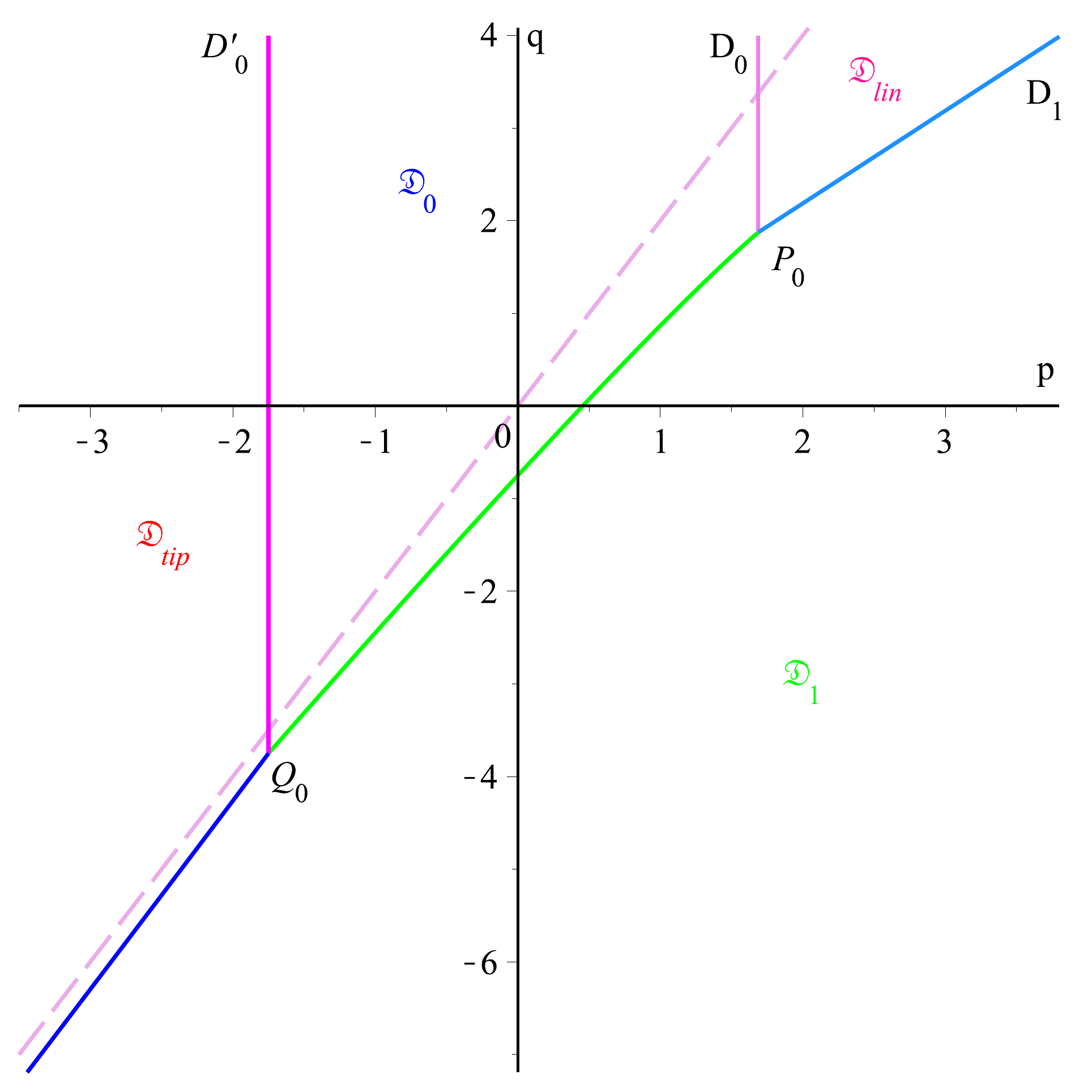}
	\caption{{\it Phase diagram of $\beta(p,q)$ in the conjecture introduced in \cite{DHLZ2015arXiv,DHLZ2018}:} The domains $\mathfrak{D}_{tip}, \mathfrak{D}_{0}, \mathfrak{D}_{lin}, \mathfrak{D}_{1}$ respectively correspond to the phases $\beta_{tip}, \beta_{0}, \beta_{lin}, \beta_{1}$. The dashed line $q=2p$, corresponding to the {\it bounded} exterior whole--plane SLE, successively passes the domains $\mathfrak{D}_{tip}, \mathfrak{D}_{0}, \mathfrak{D}_{lin}$ as was proved in \cite{BS2009}.}
	\label{fig_SLE_generalized_spectrum_diagram}
\end{figure}

This conjecture is already proved in the same works for the domain situated to the left of the right branch of the green parabola up to the intersection point with the line $D_3$: $q-p=-1-\kappa/2$, followed by the line $D_3$ up to the intersection point of this line with the right part of red parabola starting from $P_0$, followed by the red parabola up to point $P_0$, then followed by the upper half-line $D_1$: $q=p-(k^2-16)/32\kappa$ (see Figure \ref{fig_old_validity_zone}). In order to do that, the authors generalized the martingale argument in \cite{BS2009,DNNZ2015,BDZ2017} to the two parameters $p,q$ setting, then obtained the PDE \eqref{eq1zz} verified by $\mathbb{E}(\vert z \vert ^q\vert f'(z) \vert^p/\vert f(z) \vert^q)$. Since this equation is parabolic in the polar variables $(r,\theta)$, it allowed the authors to use the maximum principle to estimate the integral means of the solution on the circles $\vert z \vert =r$, for all $r$ sufficiently closed to $1$, by those of a sub--solution and a super--solution. In fact, the authors have constructed appropriate sub--solutions and super--solutions to the PDE \eqref{eq1zz} with that they proved in turn the conjecture on sub-domains of the above mentioned domain:
\begin{thm}\label{thm DHLZ}({\it Duplantier--Ho--Le--Zinsmeister} \cite{DHLZ2015arXiv,DHLZ2018})\\Let $\beta(p,q)$ the average generalized integral means spectrum of whole--plane $\SLE$. Then
	\begin{enumerate}[before=\leavevmode,label=\upshape(\Roman*),ref=\thethm (\Roman*)]
		\item\label{beta_0-tip}
		$\beta(p,q)$ is given by $\beta_0(p)$ in $\mathfrak{D}_0$ and by $\beta_{tip}(p)$ in $\mathfrak{D}_{tip}$.

		\item \label{thm_lin}
		$\beta(p,q)$ is given by $\beta_{lin}(p)$ in $\mathfrak{D}_{lin}$.

		\item\label{spec B-D-Z} Denote $\mathfrak{D}$ the interior part of the green parabola located to the left of $D_0$ and $\mathfrak{W}$, a sub-domain of $\mathfrak{D}$, the lower infinite wedge of apex $Q_0$ located between the left branch of the green parabola and the blue quartic (see Figure \ref{fig_SLE_spectrum_validity-DHLZ}). Then $\beta(p,q)$ is given by $\beta_{tip}(p,q)$ in $\mathfrak{W}$ and by $\beta_1(p,q)$ in $\mathfrak{D}\setminus \mathfrak{W}$.
		
		\item\label{spec D-N-N-Z} Denote $\hat{\mathfrak{D}}$ the racket-shaped finite domain located between the left branch of the green parabola and the right branch of the red parabola and above the line $D_3$ (see Figure \ref{fig_SLE_spectrum_validity-DHLZ}). Then $\beta(p,q)$ is given by $\beta_1(p,q)$ in $\hat{\mathfrak{D}}$.
	\end{enumerate}
\end{thm}
In the rest of the $(p,q)-$plane, they showed that $\beta_1$ is an upper bound of $\beta(p,q)$. Namely, they proved the following
\begin{prop}\label{lower bound}({\it Duplantier--Ho--Le--Zinsmeister \cite{DHLZ2015arXiv,DHLZ2018}})\\
	Let $\{\mathcal{E}_-,\mathcal{I},\mathcal{E}_+\}$ the partition of the half-plane below $\Delta_1$ by the red parabola: An interior domain $\mathcal{I}$, a left exterior wedge $\mathcal{E}_-$ of apex $T_0$ and a right exterior wedge $\mathcal{E}_+$ of the same apex (see Figure \ref{fig_E_I_E}). Then $\beta(p,q)$ is bounded below as $\beta_1(p,q)\leq \beta(p,q)$ in $\mathcal{E}_-\cup \mathcal{I}$, whereas  $\beta(p,q)\leq \beta_1(p,q)$ in $\mathcal{E}_+$.
\end{prop}
The conjecture has not been proved for the whole $(p,q)-$ plane since we lack a proof for the sharpness of the above upper bound $\beta_1$. The total domain in Theorem \ref{thm DHLZ} is so far the largest one where the values of the average generalized integral means spectrum $\beta(p,q)$ of whole--plane $\SLE$ are given and the question of determining values of $\beta(p,q)$ in the rest of the $(p,q)-$plane is still open.

\begin{figure}[httb]
	\centering
	\begin{subfigure}{0.5\textwidth}
		\centering
		\includegraphics[width=0.9\linewidth]{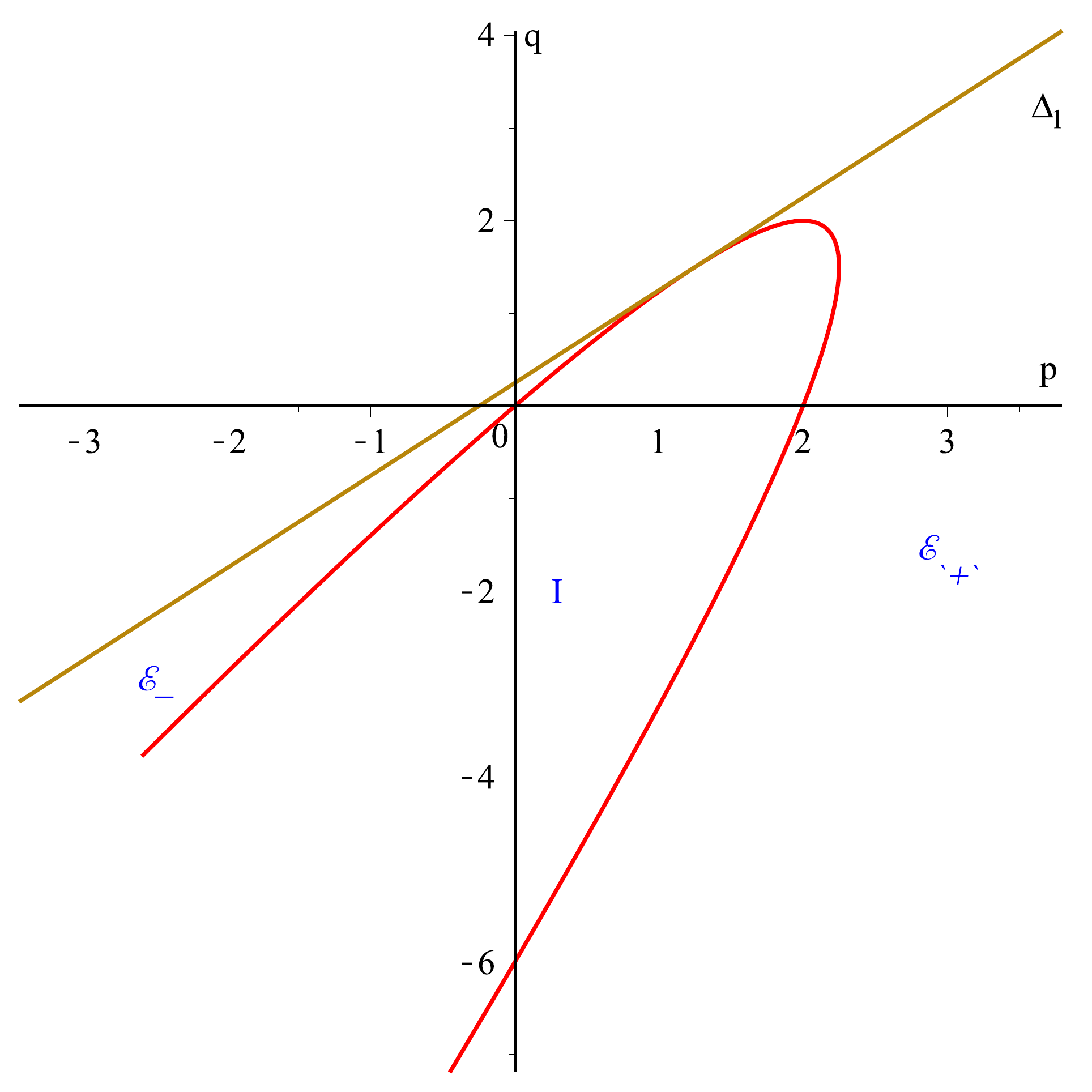}
		\caption{Domains $\mathcal{E}_{-},\mathcal{I},\mathcal{E}_{+}$}
		\label{fig_E_I_E}
	\end{subfigure}%
	\begin{subfigure}{0.5\textwidth}
		\centering
		\includegraphics[width=0.9\linewidth]{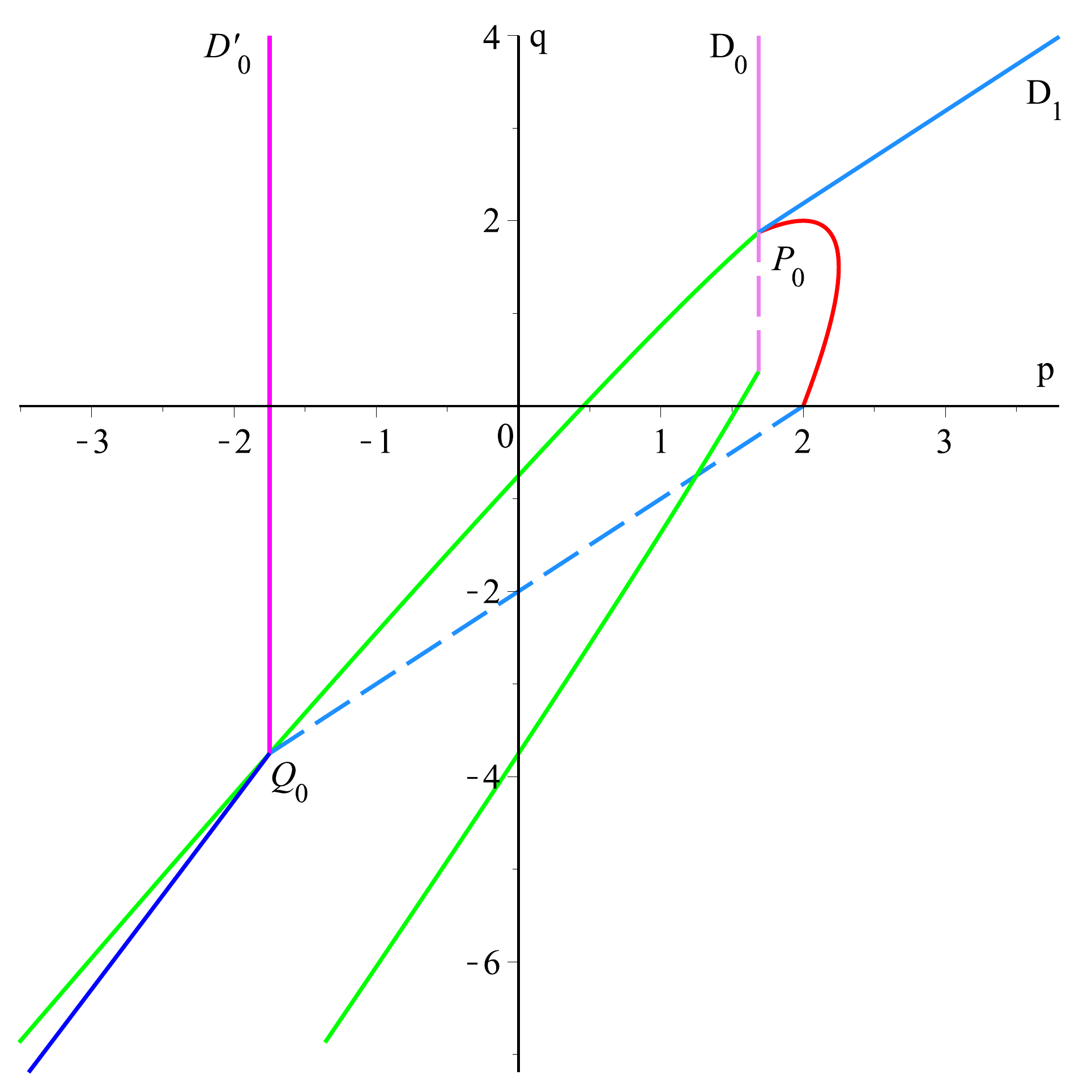}
		\caption{Domains of validity in \cite{DHLZ2015arXiv},\cite{DHLZ2018}}
		\label{fig_SLE_spectrum_validity-DHLZ}
	\end{subfigure}
	\caption{{\it Domains where the results on $\beta(p,q)$ were obtained in \cite{DHLZ2015arXiv,DHLZ2018}}}
	\label{whpl2}
\end{figure}

Finally, let us notice that, by the relation \eqref{eq_relation_gener_spec_m fold and standard_intro}, the conjecture about $\beta(p,q)$ implies another one about $\beta^{[m]}$, the average generalized integral means spectrum associated to the m--fold transform of the whole--plane SLE map $f=f_0$: Let $\mathfrak{D}_{tip}^{[m]}, \mathfrak{D}_{0}^{[m]}, \mathfrak{D}_{lin}^{[m]}, \mathfrak{D}_{1}^{[m]}$ respectively be the image of  $\mathfrak{D}_{tip}, \mathfrak{D}_{0}, \mathfrak{D}_{lin}, \mathfrak{D}_{1}$ under the inverse transformation $T_m^{-1}(p,q)=(p,(1-m)p+mq)$ of the endomorphism of $\R^2$ given by $T_m(p,q)=(p,q_m)$, where $q_m$ is defined in \eqref{eq_relation_gener_spec_m fold and standard_intro}. The value of $\beta^{[m]}$ is respectively given on $\mathfrak{D}_{tip}^{[m]}, \mathfrak{D}_{0}^{[m]}, \mathfrak{D}_{lin}^{[m]},\mathfrak{D}_{1}^{[m]}$ by $\beta_{tip},\beta_{0},\beta_{lin}$ and
\begin{equation}
\beta_1^{[m]}(p,q)=\beta_1(p,q_m)=\bigg(1+\frac{2}{m}\bigg)p-\frac{2}{m}q-\frac{1}{2}-\frac{1}{2}\sqrt{1+\frac{2\kappa}{m}(p-q)}.
\end{equation}

\begin{figure}[httb]
	\centering
	\begin{subfigure}{0.5\textwidth}
		\centering\captionsetup{width=.8\linewidth}
		\includegraphics[width=0.9\linewidth]{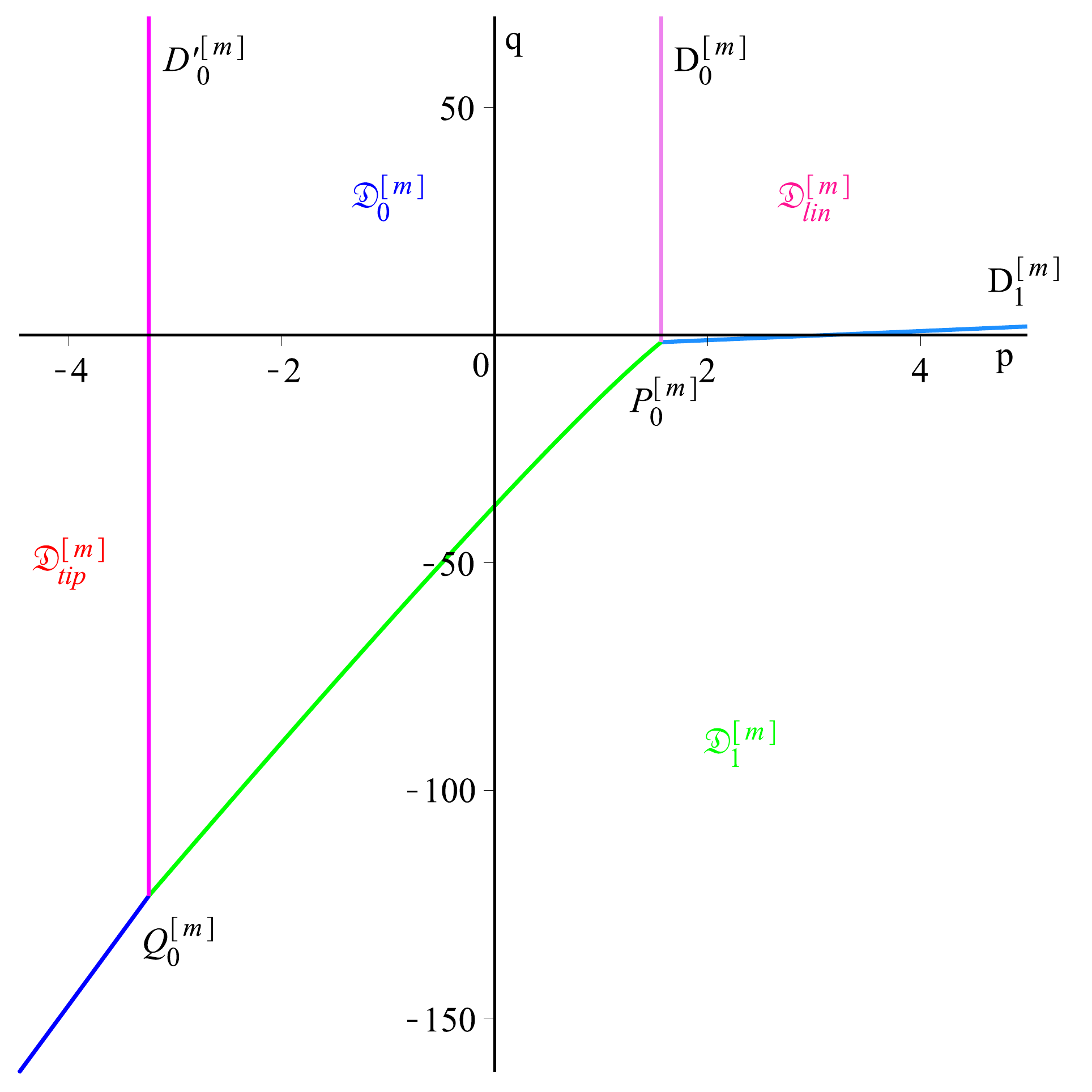}
		\caption{For $\kappa=6, m=30$, the diagram of $\beta^{[m]}(p,q)$ looks similar to that of $\beta(p,q)$. The usual integral means spectrum $\beta^{[m]}(p,0)$ of $f^{[m]}$ successively takes the phases $\beta_{tip},\beta_{0},\beta_{lin},\beta_1^{[m]}$.}
		\label{fig_m-fold_diagram_k=6}
	\end{subfigure}%
	\begin{subfigure}{0.5\textwidth}
		\centering\captionsetup{width=.8\linewidth}
		\includegraphics[width=0.9\linewidth]{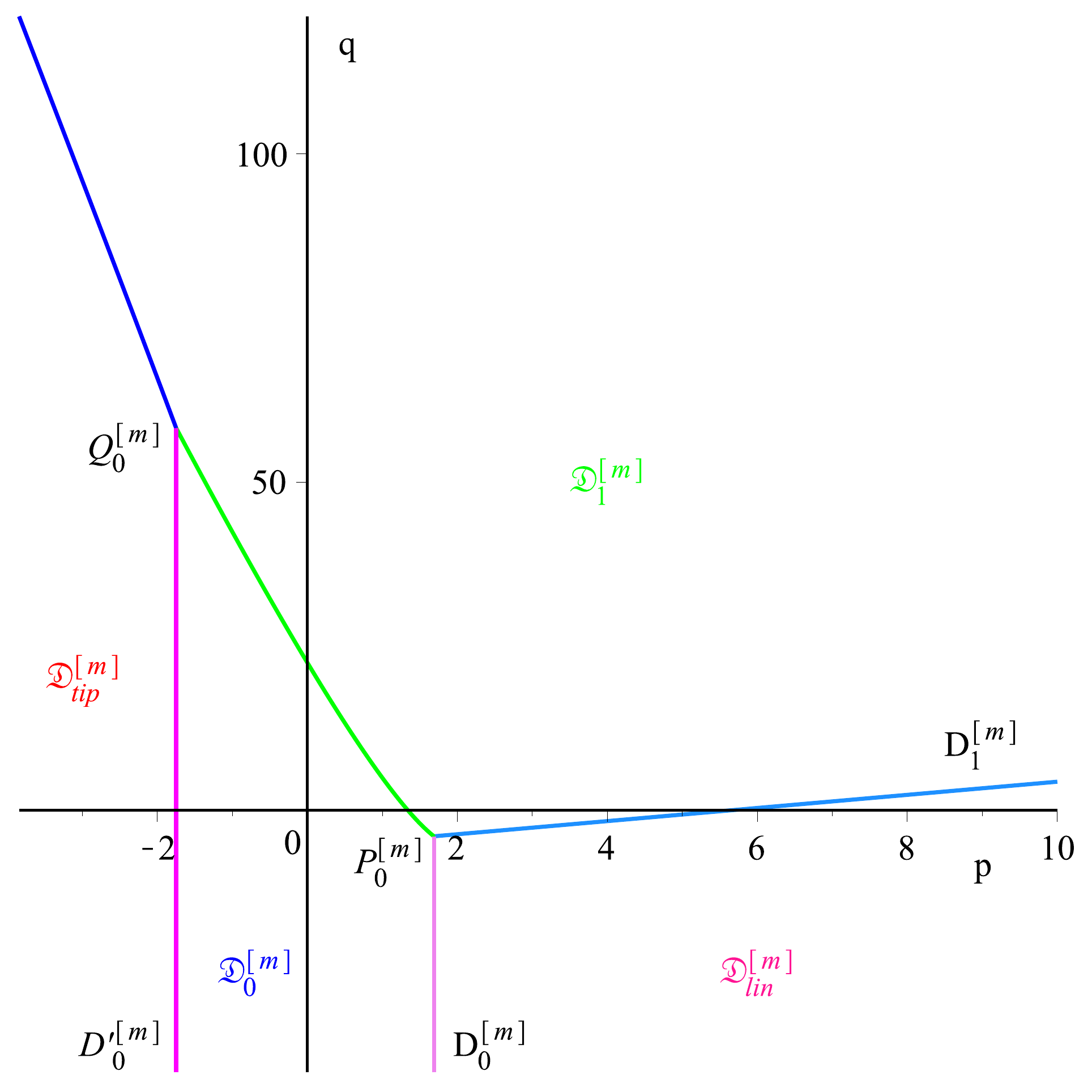}
		\caption{For $\kappa=2, m=-30$, $\det T_m^{-1}=m<0$, the vertical positions of  $\mathfrak{D}_{0}^{[m]}, \mathfrak{D}_{tip}^{[m]}, \mathfrak{D}_{lin}^{[m]},\mathfrak{D}_{1}^{[m]}$ are reversed.  The usual integral means spectrum $\beta^{[m]}(p,0)$ of $f^{[m]}$ successively takes the phases $\beta_{tip},\beta_{0},\beta_1^{[m]},\beta_{lin}$.}
		\label{fig_m-fold_diagram_k=2}
	\end{subfigure}
	\caption{{\it Conjecture about $\beta^{[m]}(p,q)$}}
	\label{whpl2}
\end{figure}

Theorem \ref{thm DHLZ} shows that the conjecture about $\beta^{[m]}$ is confirmed on the image domain under the linear transformation $T_m^{-1}(p,q)=(p,(1-m)p+mq)$ of the validity domain of the conjecture about $\beta(p,q)$ (see Figure \ref{fig_m-fold_old_validity_zone}).

\section{Proof of the main theorem}
\subsection{Maximum principle method}
Let us first introduce the following notion of {\it integral means exponent} that will be often mentioned in this work: For $\phi(z,\bar{z})$ a function defined in $\D$ and integrable on any circle $r\partial \D$, where $0<r<1$, the integral means exponent of $\phi$ is the number $\epsilon_\phi$ defined as
\begin{equation}
\epsilon_\phi:=\limsup_{r\to 1^-}\frac{\log\int_{r\partial\D}\vert \phi(z,\bar{z}) \vert \vert dz \vert}{-\log(1-r)}.
\end{equation}
The generalized spectrum $\beta(p,q)$ is then the integral means exponent of the moment expectation $\mathbb{E}(\vert z \vert ^q\vert f'(z) \vert^p/\vert f(z) \vert^q)$.
The problem of determining the values of $\beta(p,q)$ takes us to the question of estimating the integral means exponent of the solution of Eq. \eqref{eq1zz}. As mentioned above, the operator $\mathcal{P}(D)$ when written in polar coordinates as in \eqref{eq_main_polar}, is {\it parabolic}. Suppose that $\phi_{-},\phi_{+}$ are positive, bounded functions on the circle of radius $r_0<1$ such that $\mathcal{P}(D)(\phi_{-})<0$ and $\mathcal{P}(D)(\phi_{+})>0$. One can find positive constants $m,M$ such that $m\phi_{+}<G<M\phi_{-}$ on the circle $r=r_0$. The {\it maximum principle} and the {\it minimum principle} (\cite{eva1998}, Theorem 7.1.9) yield that $m\phi_{+}<G<M\phi_{-}$ in the annulus $r_0<r<1$ (we say that $\phi_{-},\phi_{+}$ are respectively a {\it sub-solution} and a {\it super-solution} of \eqref{eq_main_polar}, or of the differential operator $\mathcal{P}(D)$). It implies that $\epsilon_{\phi_{+}}\leq\epsilon_G\leq\epsilon_{\phi_{-}}$ where $\epsilon_{\phi_{+}},\epsilon_G,\epsilon_{\phi_{-}}$ are respectively the integral means exponents of $\phi_{+},G,\phi_{-}$. Thus we have the following estimations on $\beta(p,q)$: $\epsilon_{\phi_{+}}\leq\beta(p,q)\leq\epsilon_{\phi_{-}}$. Indeed, we will construct appropriate sub-solutions and super-solutions to estimate the generalized spectrum $\beta(p,q)$ and prove Theorem \ref{thm_main_theorem}.
\subsection{Sub--solutions and super--solutions}\label{sec_normalized-testf}
Inspired by the previous works \cite{DNNZ2015,DHLZ2018}, let us first introduce the following test functions
\begin{equation}\label{testf_in_g}
\psi(z,\bar{z})=(1-z\bar{z})^{-\beta(\gamma)}g(u),
\end{equation}
where $\gamma$ is a real parameter, $\beta(\gamma)$ is defined by \eqref{eq_beta}, $u=(1-z)(1-\bar{z})$ and $g$ is a general solution of the equation
\begin{equation}\label{eq_g}
[p(2-\sigma u)-2\beta(\gamma)]g(u)+\bigg[ \frac{\kappa}{2}(2-\kappa)-(4-u) \bigg]ug'(u)+\frac{\kappa}{2}(4-u)u^2g''(u)=0,
\end{equation}
with $\sigma:=q/p-1$. The generalized spectrum $\beta(p,q)$ has been studied in \cite{DHLZ2018} by using Maximum principle method with sub--solutions and super--solutions to \eqref{eq1zz} of the following form
\begin{equation}\label{eq log modification}
	\phi_{\pm}(z,\bar{z})=\psi(z,\bar{z})(-\log(1-z\bar{z}))^{\delta_{\pm}}.
\end{equation}
Here we will construct {\it mixed sub--solutions} and {\it mixed super--solutions} as
\begin{equation}\label{eq def mixedf}
\phi_{\pm}(z,\bar{z})=[\psi_0(z,\bar{z})+\psi_1(z,\bar{z})](-\log(1-z\bar{z}))^{\delta_{\pm}}
\end{equation}
where $\psi_0,\psi_1$ are tests functions of the form \eqref{testf_in_g} corresponding to two different values of parameter $\gamma$. We notice that the functions $\phi_{\pm}$ \eqref{eq log modification}, \eqref{eq def mixedf} are continuous on any circle of radius $r_0<1$, so they are bounded on this circle. One then needs to consider the positivity of $\phi_{\pm}$ and the sign of $\mathcal{P}(D)(\phi_{\pm})$ to construct sub--solutions and super--solutions to \eqref{eq_main_polar}.\\

The above Eq. \eqref{eq_g} is called {\it boundary equation} since it comes from the restriction of Eq. \eqref{eq_main_polar} to the unit circle $r=1$. Let us give a brief explain for the appearance of Eq. \eqref{eq_g} and the exponent $\beta(\gamma)$ in \eqref{testf_in_g}. This equation appears when we look for necessary conditions for that \eqref{eq1zz} has a solution of the form $G(z,\bar{z})=(1-z\bar{z})^{-\beta}g(\vert 1-z \vert ^2)$. Namely, after substituting this form of $G$ into Eq. \eqref{eq_main_polar} and then setting $r=1$, we arrive at
\begin{align}
	[2p-p\sigma (2-2\cos\theta)-2\beta]g&-[4\sin^2\theta- \kappa(2-2\cos\theta)\cos\theta]g'\label{eq to unit circle g}\\
	&+2\kappa(2-2\cos\theta)\sin^2\theta g''=0.\nonumber
\end{align}
Denote $x:=2-2\cos\theta$, Eq. \eqref{eq to unit circle g} becomes
\begin{equation}\label{eq to unit circle g-x defined}
[p(2-\sigma x)-2\beta]g(x)+\bigg[\frac{\kappa}{2}(2-x)-(4-x)\bigg]xg'(x)+\frac{\kappa}{2}(4-x)x^2g''(x)=0.
\end{equation}
If we set
\begin{equation}
	g(u)=u^\gamma g_0(u),
\end{equation}
then $g_0(u)$ verifies the following equation
\begin{align}
	[A^{\sigma}(\gamma)-2\beta+2\beta(\gamma)]g_0(u)&+\left[  \frac{\kappa}{2}(2-u)+(\kappa\gamma-1)(4-u)\right]g'_0(u)\label{eq gamma0 before set beta}\\
	&+\frac{\kappa}{2}(4-u)ug''_0(u)=0.\nonumber
\end{align}
For $\beta=\beta(\gamma)$, Eq. \eqref{eq gamma0 before set beta} becomes a {\it hypergeometric equation}
\begin{equation}\label{eq hypergeom g0}
	A^{\sigma}(\gamma)g_0(u)+\left[  \frac{\kappa}{2}(2-u)+(\kappa\gamma-1)(4-u)\right]g'_0(u)+\frac{\kappa}{2}(4-u)ug''_0(u)=0,
\end{equation}
and Eq. \eqref{eq to unit circle g-x defined} becomes \eqref{eq_g}. Indeed, the change of variable $u:=x/4$ implies that $g_0$ is a solution of the standard hypergeometric equation
\begin{equation}
	-ab g_0(x)+\left[ c-(a+b+1)x\right]g'_0(x)+x(1-x)g''_0(x)=0,
\end{equation}
where
\begin{align}
	&a=\gamma-\gamma_1,\,\, b=\gamma-\gamma_1^-,\,\,c=1+\gamma-\gamma',\label{eq a,b,c}\\
	&a'=\gamma'-\gamma_1,\,\, b'=\gamma'-\gamma_1^-,\,\,c'=1+\gamma'-\gamma.\label{eq a',b',c'}
\end{align}
Here $\gamma$ and $\gamma'$ are dual values of the parameter defined by \eqref{eq def dual points}. 
The general solution of Eq. \eqref{eq hypergeom g0} is
\begin{equation}
	g_0(u)=C_1{}_2F_1\bigg(a,b,c,\frac{u}{4}\bigg)+C_2\bigg(\frac{u}{4}\bigg)^{\gamma'-\gamma}{}_2F_1\bigg(a',b',c',\frac{u}{4}\bigg),
\end{equation}
Hence we have the expression of $g$ as
\begin{equation}\label{g}
g(u)=C_1\bigg(\frac{u}{4}\bigg)^{\gamma}{}_2F_1\bigg(a,b,c,\frac{u}{4}\bigg)+C_2\bigg(\frac{u}{4}\bigg)^{\gamma'}{}_2F_1\bigg(a',b',c',\frac{u}{4}\bigg).
\end{equation}
\begin{rem}
	Eqs. \eqref{g},\eqref{eq a,b,c},\eqref{eq a',b',c'} show that $\gamma$ and $\gamma'$ are symmetric in the expression of $g$. If we set $\beta=\beta(\gamma')$ and consider
	\begin{equation}
		g(u)=u^{\gamma'}g_1(u),
	\end{equation}
	then $g_1$ is a solution of the hypergeometric equation
	\begin{equation}
		A^{\sigma}(\gamma')g_1(u)+\left[  \frac{\kappa}{2}(2-u)+(\kappa\gamma'-1)(4-u)\right]g'_1(u)+\frac{\kappa}{2}(4-u)ug''_1(u)=0\label{eq hypergeom g0_prime}.
	\end{equation}
	Thus we have the expression of $g_1$ as
	\begin{equation}
		g_1(u)=C_1\bigg(\frac{u}{4}\bigg)^{\gamma-\gamma'}{}_2F_1\bigg(a,b,c,\frac{u}{4}\bigg)+C_2 {}_2F_1\bigg(a',b',c',\frac{u}{4}\bigg).
	\end{equation}
\end{rem}

\subsection{Singularity and sign of test functions}
In our constructions of sub--solutions and super--solutions by means of \eqref{testf_in_g},\eqref{eq log modification},\eqref{eq def mixedf}, we want that $g(u)$ is singularity free and should have a second derivative everywhere on the unit circle except at the point $z=1$ (the equation \eqref{eq1zz} on $G$ has singularity at that point). Note that $z=-1$ corresponds to $u=4$, hence $g(u)$ should have expansion $g(u)=c+O(4-u)$ at the endpoint 4. We also want the functions $\psi$ \eqref{testf_in_g}  to be positive on $\D$ (or at least $\psi>0$ in a neighborhood of $z=1$), i.e $g\geq0$ on $(0,4)$ (or at least $g>0$ in a neighborhood of $u=0$).\\

{\it Smoothness of $g$ at $u=4$.} One has the following developments at $u=4$ of the hypergeometric functions in \eqref{g}
\begin{align}
	&{}_2F_1\bigg(a,b,c,\frac{u}{4}\bigg)=\frac{\sqrt{\pi}\Gamma(c)}{\Gamma(\frac{1}{2}+a)\Gamma(\frac{1}{2}+b)}-\frac{\sqrt{\pi}\Gamma(c)}{\Gamma(a)\Gamma(b)}\sqrt{4-u}+O(4-u),\\
	&{}_2F_1\bigg(a',b',c',\frac{u}{4}\bigg)=\frac{\sqrt{\pi}\Gamma(c')}{\Gamma(\frac{1}{2}+a')\Gamma(\frac{1}{2}+b')}-\frac{\sqrt{\pi}\Gamma(c')}{\Gamma(a')\Gamma(b')}\sqrt{4-u}+O(4-u).
\end{align}
We then conclude that the function $g$ \eqref{g} does not have singularities at $u=4$ if and only if
\begin{equation}\label{cond_regular_at_4}
	C_1\frac{\sqrt{\pi}\Gamma(c)}{\Gamma(a)\Gamma(b)}+C_2 \frac{\sqrt{\pi}\Gamma(c')}{\Gamma(a')\Gamma(b')}=0.
\end{equation}
By excluding the trivial case $C_1=C_2=0$, Eq. \eqref{cond_regular_at_4} turns out to be one of the followings:

{\it Case I: $C_1=0$, $a'=-n$ or $b'=-n$ where $n\in \N$.} Function $g(u)$ is given by
\begin{align}
	&g(u)=C_2\bigg(\frac{u}{4}\bigg)^{\gamma_1-n}{}_2F_1\bigg(-n,\gamma_1-\gamma_1^{-} -n,\frac{1}{2}-\frac{2}{\kappa}-2n+2\gamma_1,\frac{u}{4}\bigg)\\
	\text{or } &g(u)=C_2\bigg(\frac{u}{4}\bigg)^{\gamma_1^{-}-n}{}_2F_1\bigg(\gamma_1^{-}-\gamma_1 -n,-n,\frac{1}{2}-\frac{2}{\kappa}-2n+2\gamma_1^{-},\frac{u}{4}\bigg).
\end{align}
Note that the above hypergeometric functions are polynomials of order $n$.

{\it Case II: $C_2=0$, $a=-n$ or $b=-n$ where $n\in \N$.} Function $g(u)$ is given by
\begin{align}
	&g(u)=C_1\bigg(\frac{u}{4}\bigg)^{\gamma_1-n}{}_2F_1\bigg(-n,\gamma_1-\gamma_1^{-} -n,\frac{1}{2}-\frac{2}{\kappa}-2n+2\gamma_1,\frac{u}{4}\bigg)\\
	\text{or } &g(u)=C_1\bigg(\frac{u}{4}\bigg)^{\gamma_1^{-}-n}{}_2F_1\bigg(\gamma_1^{-}-\gamma_1 -n,-n,\frac{1}{2}-\frac{2}{\kappa}-2n+2\gamma_1^{-},\frac{u}{4}\bigg).
\end{align}

{\it Case III: $a=-n,a'=-m$ where $n,m\in \N$.} Function $g(u)$ is given by
\begin{align}
	g(u)=&C_1\bigg(\frac{u}{4}\bigg)^{\gamma_1-n}{}_2F_1\bigg(-n,\gamma_1-\gamma_1^{-} -n,\frac{1}{2}-\frac{2}{\kappa}-2n+2\gamma_1,\frac{u}{4}\bigg)\\
	+&C_2\bigg(\frac{u}{4}\bigg)^{\gamma_1-m}{}_2F_1\bigg(-m,\gamma_1-\gamma_1^{-} -m,\frac{1}{2}-\frac{2}{\kappa}-2m+2\gamma_1,\frac{u}{4}\bigg).\nonumber
\end{align}
A necessary and sufficient condition for that function $g$ \eqref{g} has the above form is that $\gamma+\gamma'=1/2+2/\kappa$, here, $\gamma=\gamma_1-n, \gamma'=\gamma_1-m$. This is equivalent to
\begin{equation}
	q=p-\frac{[2(n+m)+1]^2\kappa^2-16}{32\kappa}.
\end{equation}
We note that the case $b=-n,b'=-m$ does not happen since it implies an absurd: $1/2+2\sqrt{1+2\kappa(p-q)}/\kappa=\gamma+\gamma'-2\gamma_1^-=b+b'=-n-m$. The case $a=-n, b'=-m$ does not happen since it implies that $-1/2=\gamma+\gamma'-\gamma_1-\gamma_1^{-}=a+b'=-n-m$. Similarly, the case $a'=-n, b=-m$ does not happen.

{\it Case IV: $C_1,C2$ are non-null and $a,b,a',b'$ are not non-positive integers.}\\
We have
\begin{equation}\label{eq C_0}
	\frac{C_2}{C_1}=C_0:=-\frac{\Gamma(c)}{\Gamma(a)\Gamma(b)}\frac{\Gamma(a')\Gamma(b')}{\Gamma(c')}.
\end{equation}

{\it Positivity of $g$.} We now give values of $C_1,C_2$ for that the function $g(u)$ is positive in a neighborhood of $u=0$:
In the above Case I, we take $C_2>0$. In Case II, we take $C_1>0$. In Case III and Case IV, if $\gamma<\gamma'$ we take $C_1>0$, if $\gamma'<\gamma$ we take $C_2>0$.\\

The following lemma provides a sufficient condition for the positivity (more precisely, non-negativity and not identically null) of $\psi$ on $\D$
\begin{lem}\label{lem_g_positive}
	If $\min(\gamma,\gamma')<\gamma_1'$ then $\psi$ is positive on the unit disc $\D$.
\end{lem}
\begin{proof}
	Without loss of generality, we assume that $\gamma=\min(\gamma,\gamma')$. It is sufficient to prove that if $\gamma<\gamma_1'$ then $g_1\geq0$ on $[0,4]$.
	
	First, let us show that inequality $\gamma<\gamma_1'$ implies that $g_1(4)\geq0$. We have
	\begin{equation}
		g_1(4)=C_1\pi^{-\frac{3}{2}}\cos[\pi(a+b)]\Gamma\bigg(\frac{1}{2}+a+b\bigg)\Gamma\bigg(\frac{1}{2}-a\bigg)\Gamma\bigg(\frac{1}{2}-b\bigg).
	\end{equation}
	Since $\gamma<\gamma'$, one has that $\gamma<1/\kappa+1/4$. Then
	\begin{equation}\label{a+b}
		a+b=2\gamma-\gamma_1-\gamma_1^{-}=2\gamma-\frac{2}{\kappa}<\frac{1}{2}.
	\end{equation}
	We then always have that
	\begin{equation}
		\cos[\pi(a+b)]\Gamma\bigg(\frac{1}{2}+a+b\bigg)\geq0.
	\end{equation}
	We also notice that $a<b$. Inequality \eqref{a+b} then implies that $a<1/4$, hence
	 $$\Gamma(1/2-a)>0.$$As a consequence, if $\Gamma(1/2-b)$ is positive then $g_1(4)$ is positive. We now note that
	\begin{equation}
		\frac{1}{2}-b=\frac{1}{2}-\gamma+\gamma_1^{-}=\frac{1}{2}+\frac{2}{\kappa}-\gamma-\gamma_1=\gamma_1'-\gamma.
	\end{equation}
	Therefore if $\gamma<\gamma_1'$ then $1/2-b>0$, hence $g_1(4)\geq0$.
	
	One also notices that inequality $\gamma<\gamma_1'$ implies that $\gamma'>\gamma_1$, hence $A^{\sigma}(\gamma')<0$. In summary, $g_1$ is a solution of the hypergeometric equation \eqref{eq hypergeom g0_prime} in which the sign of the coefficient of $g_1$ is negative while the sign of the coefficient of $g_1''$ is positive. Moreover, $g_1(0)>0$ and $g_1(4)\geq0$. Suppose that $g_1$ has a negative local minimum on $(0,4)$. Then at this point, one has $g_1< 0, g_1'=0, g_1''\geq0$. This is contradictory with the signs of the coefficients of $g_1$ and $g_1''$ in \eqref{eq hypergeom g0_prime}. Therefore, $g_1$ is nonnegative on $[0,4]$. Equivalently, $\psi$ is positive on $\D$.\\
\end{proof}
{\bf Standard test functions.} Without loss of generality, we will hereafter use the following standard form of functions $\psi$ \eqref{testf_in_g}
	\begin{equation}\label{testf}
	\psi(z,\bar{z})=(1-z\bar{z})^{-\beta(\gamma)}u^{\gamma}g_0(u),
	\end{equation}
	where $g_0(0)=1$ and $g_0$ does not have singularities at $u=4$.\\
	
The above arguments imply that the set of standard test functions is consist of functions \eqref{testf}, where:\\
\begin{equation}\label{eq_gamma1-n}
	\gamma=\gamma_1-n,n\in\N \text{ and } g_0(u)={}_2F_1\bigg(-n,\gamma_1-\gamma_1^{-} -n,\frac{1}{2}-\frac{2}{\kappa}-2n+2\gamma_1,\frac{u}{4}\bigg),
\end{equation}
\begin{equation}
	\gamma=\gamma_1^{-}-n,n\in\N \text{ and } g_0(u)={}_2F_1\bigg(\gamma_1^{-}-\gamma_1 -n,-n,\frac{1}{2}-\frac{2}{\kappa}-2n+2\gamma_1^{-},\frac{u}{4}\bigg),
\end{equation}
\begin{align}
	\gamma=\gamma_1-n &\text{ and } g_0(u)={}_2F_1\bigg(-n,\gamma_1-\gamma_1^{-} -n,\frac{1}{2}-\frac{2}{\kappa}-2n+2\gamma_1,\frac{u}{4}\bigg)\\
	 &+C\bigg(\frac{u}{4}\bigg)^{n-m}{}_2F_1\bigg(-m,\gamma_1-\gamma_1^{-} -m,\frac{1}{2}-\frac{2}{\kappa}-2m+2\gamma_1,\frac{u}{4}\bigg),\nonumber\\
	\text{with }m<n,  &q=p-\frac{[2(n+m)+1]^2\kappa^2-16}{32\kappa}; m,n\in\N,\nonumber
\end{align}
\begin{align}\label{eq_g_0_normalized}
	&\gamma\text{ such that } \gamma_1-\gamma,\gamma^{-}_1-\gamma,\gamma_1-\gamma',\gamma^{-}_1-\gamma'\notin\N, \gamma<\gamma'\text{ and }\\
	&g_0(u)={}_2F_1\bigg(a,b,c,\frac{u}{4}\bigg)+C_0\bigg(\frac{u}{4}\bigg)^{\gamma'-\gamma}{}_2F_1\bigg(a',b',c',\frac{u}{4}\bigg),\nonumber
\end{align}
where $C_0$ is defined by \eqref{eq C_0}.\\

Let us give some notices on this set:

First, the definition of standard test function does not include the positivity of $\psi$ on $\D$. We only have that $\psi(0)>0$ in a neighborhood of $z=1$. However, Lemma \ref{lem_g_positive} gives a sufficient condition for that $\psi$ is positive on $\D$.

Second, the functions $g_0$ in \eqref{testf} has the following properties:
\begin{align}
	&\frac{ug_0'}{g_0}\rightarrow 0,\,u\rightarrow 0,\label{g0'/g0 at 0}\\
	&\frac{g_0'(4)}{g_0(4)}=\frac{1}{\kappa}A^{\sigma}(\gamma).\label{g0'/g0 at 4}
\end{align}
The identity \eqref{g0'/g0 at 4} is  derived from \eqref{eq hypergeom g0} and the smoothness of $g_0$ at 4. The properties \eqref{g0'/g0 at 0}, \eqref{g0'/g0 at 4} play an important role in the next section where we study the action of the differential $\mathcal{P}(D)$ on the standard test functions and their logarithmic modifications.

Third, a standard test function $\psi$ can be written as $\psi(z,\bar{z})=(1-z\bar{z})^{-\beta(\gamma)}g(u)$ where $g$ is defined by \eqref{g} as a linear combination of
\begin{equation}
	\bigg(\frac{u}{4}\bigg)^{\gamma}{}_2F_1\bigg(a,b,c,\frac{u}{4}\bigg) \text{ and } \bigg(\frac{u}{4}\bigg)^{\gamma'}{}_2F_1\bigg(a',b',c',\frac{u}{4}\bigg).
\end{equation}
 Between these two functions, the leading term in the $u\rightarrow 0$ limit is that of the exponent $\min(\gamma,\gamma')$. It is important to note that the functions defined by \eqref{testf},\eqref{eq_gamma1-n} are the only in the set of standard test functions such that this leading term may not appear in their formulas. In other words, among the standard test functions \eqref{testf}, these functions are the only that may have $\gamma > \gamma'$. Namely, in the domain located below the line of slope 1: $q=p-[\kappa^2(1+4n)^2-16]/32\kappa$, one has $\gamma=\gamma_1-n>\gamma_1'+n=\gamma'$. We will come back to this observation in Section \ref{sec_mixed}.

\subsection{Action of the differential operator}\label{Action of the operator} In this section, we study the action of the differential operator $\mathcal{P}(D)$ on the standard test functions \eqref{testf} and on their logarithmic modifications \eqref{eq log modification}.

{\it Action of the differential operator on the standard test functions.} Let $\psi(z,\bar{z})$ be a standard test function \eqref{testf}. Since $g_0$ verifies Eq. \eqref{eq hypergeom g0}, we have that
\begin{align}\label{action_oper}
	&\frac{\mathcal{P}(D)\left[\psi(z,\bar{z})\right]}{\psi(z,\bar{z})}\\ \nonumber
	&=(1-z\bar{z})\bigg[-\frac{1}{u}\left[C(\gamma)+A(\gamma)\right]+\frac{1}{4-u}2A^{\sigma}(\gamma)+\bigg(\frac{\kappa}{2}-1-\frac{2\kappa}{4-u}\bigg)\frac{g_0'}{g_0}\bigg]\\
	&+\frac{(1-z\bar{z})^2}{u^2}\bigg\{\frac{1}{4-u}\bigg[4A^\sigma (\gamma)-\kappa u \frac{g'_0}{g_0}\bigg]+(\sigma-1)p+\bigg(\frac{\kappa}{2}+1\bigg)\bigg(\gamma+u\frac{g'_0}{g_0}\bigg)\bigg\}.\nonumber
\end{align}
The identity \eqref{g0'/g0 at 4} resolves the apparent singularities at $u=4$ in \eqref{action_oper}. In the $u\to 0$ limit,
the term of $(1-z\bar{z})^2/u^2$ is the most singular term, hence it is the leading term near $z=1$. Since \eqref{g0'/g0 at 0}, the $(1-z\bar{z})^2/u^2$ term is equivalent to
\begin{align}
	\frac{(1-z\bar{z})^2}{u^2}\bigg\{ A^{\sigma}(\gamma)+(\sigma-1)p+\bigg(\frac{\kappa}{2}+1\bigg)\gamma \bigg\}=\frac{(1-z\bar{z})^2}{u^2}C(\gamma).
\end{align}

{\it Action of the differential operator on the logarithmic modifications.} Let us now consider the action of the differential operator \eqref{eq1zz} on a logarithmic modification of the standard test functions $\psi$ \eqref{testf}, that is of the form
\begin{equation}
	\psi(z,\bar{z}) l_\delta(z\bar{z}),
\end{equation}
 where $l_\delta(z\bar{z}):=(-\log(1-z\bar{z}))^\delta,\delta\in\R$.\\
 The advantage of the factor $l_\delta$ is based on the fact that it does not change the integral means exponent of $\psi$ while it may add to the action of the differential operator \eqref{eq1zz} a new leading term as $r\to 1$ and the sign of this term can be easily controlled.  One has that
\begin{equation}
	\mathcal{P}(D)[\psi(z,\bar{z})l_\delta(z\bar{z})]=l_\delta(z\bar{z})\bigg\{ \mathcal{P}(D)[\psi(z,\bar{z})] -\psi(z,\bar{z})\frac{2\delta z \bar{z}u^{-1}}{[-\log(1-z\bar{z})]}\bigg\}.
\end{equation}
Together with \eqref{action_oper}, it implies that one can write ${\mathcal{P}(D)[\psi(z,\bar{z})l_\delta(z\bar{z})]}/\psi(z,\bar{z})l_\delta(z\bar{z})]$ as the sum (up to small order terms as $r\to 1$) of three terms
\begin{equation}
	\frac{\mathcal{P}(D)[\psi(z,\bar{z})l_\delta(z\bar{z})]}{\psi(z,\bar{z})l_\delta(z\bar{z})}=\frac{(1-z\bar{z})^2}{u^2}+\frac{1-z\bar{z}}{u}-\frac{2\delta z \bar{z}}{u[-\log(1-z\bar{z})]}
\end{equation}
For a reason of succinctness, the coefficients of the first two terms were omitted in the above expression. We study the sign of $\mathcal{P}(D)(\psi l_\delta)$ near the unit circle $r=1$ in three cases:\\

{\it Case I: u is bounded away from 0.} The logarithmic term is the leading term. Then the sign of $\mathcal{P}(D)(\psi l_\delta)$ is that of this term and opposite to the sign of $\delta$.\\

{\it Case II: $1-z\bar{z}\leq u^{1/(2-\epsilon)}$, where $0<\epsilon<2$.} One has
\begin{align}
	\frac{1-z\bar{z}}{u} &\lessapprox \frac{1}{u[-\log(1-z\bar{z})]}\\
	\frac{(1-z\bar{z})^2}{u^2}\leq \frac{(1-z\bar{z})^2}{u}&\frac{1}{(1-z\bar{z})^{2-\epsilon}}=\frac{(1-z\bar{z})^{\epsilon}}{u}\lessapprox \frac{1}{u[-\log(1-z\bar{z})]}.
\end{align}
Then the logarithmic term is the leading term.\\

{\it Case III: $u^{1/(2-\epsilon)}<1-z\bar{z}$, where $0<\epsilon<2$.} For a sufficiently small $\epsilon$,
\begin{align}
	\frac{(1-z\bar{z})^2}{u^2}\geq \frac{(1-z\bar{z})^2}{u}&\frac{1}{(1-z\bar{z})^{2-\epsilon}}=\frac{(1-z\bar{z})^{\epsilon}}{u}\gtrapprox 	\frac{1-z\bar{z}}{u}.
\end{align}
Moreover, one always has
\begin{equation}
	\frac{1}{u[-\log(1-z\bar{z})]}\gtrapprox \frac{1-z\bar{z}}{u}.
\end{equation}
As mentioned above, the $(1-z\bar{z})^2/u^2$ term is equivalent to $C(\gamma)(1-z\bar{z})^2/u^2$ as $u\to 0$ (i.e $z\to 1$). Therefore if $C(\gamma)=0$, equivalently $\gamma=\gamma_0$, then the leading term is the logarithmic one. Otherwise, if we take $\delta$ such that the sign of the logarithmic term is the same as that of $C(\gamma)$ then the sign of  $\mathcal{P}(D)(\psi l_\delta)$ near $z=1$ is that of $C(\gamma)$. We arrive at the following
\begin{prop}\label{prop single testf}
	Let $\psi_0$ be a standard test function \eqref{testf}. Then
	\begin{enumerate}[before=\leavevmode,label=\upshape(\alph*),ref=\thethm (\Roman*)]
		\item\label{prop single testf gamma_0} If $\gamma=\gamma_0$ then there exists $\delta_1>0,\delta_2<0$ and $r_0<1$ such that in the annulus $r_0<r<1$, $
		 	\mathcal{P}(D)[\psi(z,\bar{z})l_{\delta_1}(z\bar{z})]<0$ and  $\mathcal{P}(D)[\psi(z,\bar{z})l_{\delta_2}(z\bar{z})]>0.$
	 \item If $C(\gamma)<0$ then there exists $\delta>0$ and $r_0<1$ such that in the annulus $r_0<r<1$, $\mathcal{P}(D)[\psi(z,\bar{z})l_{\delta}(z\bar{z})]<0$.
	 \item If $C(\gamma)>0$ then there exists $\delta<0$ and $r_0<1$ such that in the annulus $r_0<r<1$, $\mathcal{P}(D)[\psi(z,\bar{z})l_{\delta}(z\bar{z})]>0$.
	\end{enumerate}
\end{prop}
Let us note that the authors used the ideas of Proposition \ref{prop single testf} in \cite{DHLZ2015arXiv,DHLZ2018} to construct sub--solutions and super--solutions of the form \eqref{eq log modification} to Eq. \eqref{eq1zz} and then proved Proposition \ref{lower bound} and Theorem \ref{beta_0-tip}, \ref{thm_lin}, \ref{spec D-N-N-Z}. Namely, Proposition \ref{prop single testf}(b), \ref{prop single testf}(c) and the set up $\gamma=\gamma_1$ were used for the proof of Proposition \ref{lower bound}. Proposition \ref{prop single testf}(a) was used for the proof of Theorem \ref{beta_0-tip}. Proposition \ref{prop single testf}(b) and the convexity of the spectrum function $\beta(p,q)$ were used for the proof of Theorem \ref{thm_lin}. Proposition \ref{prop single testf}(b) and Proposition \ref{lower bound} were used for the proof of Theorem \ref{spec D-N-N-Z}.

\subsection{Mixed test function}\label{sec_mixed}
For sub--solutions and super--solutions \eqref{eq log modification} constructed from a single standard test function \eqref{testf}, the three important factors of Maximum principle method (positivity, integral means exponent of the sub--solutions or super--solutions and the sign of the action of the operator on these functions) are determined by only one parameter $\gamma$. In this section, we develop a new idea to apply Maximum principle method in which the single standard test function \eqref{testf} is replaced by a {\it mixed test function}. Precisely, we consider test functions of the form
\begin{equation}\label{mixed testf}
	\psi(z,\bar{z}):=\psi_0+\psi_1=(1-z\bar{z})^{-\beta(\gamma)}u^{\gamma}g_0(u)+(1-z\bar{z})^{-\beta(\tilde{\gamma})}u^{\tilde{\gamma}}\tilde{g}_0(u),
\end{equation}
where $\psi_0, \psi_1$ are single standard test functions \eqref{testf} and construct sub--solutions and super-solutions of the form
\begin{equation}\label{eq_mixed_sub_super_solutions}
	\phi_{\pm}(z,\bar{z}):=\psi(z,\bar{z})l_{\delta_{\pm}}(z\bar{z})=[\psi_0(z,\bar{z})+\psi_1(z,\bar{z})]l_{\delta_{\pm}}(z\bar{z}),
\end{equation}
where  $l_{\delta_{\pm}}(z\bar{z})=(-\log(1-z\bar{z}))^{\delta_\pm}$.
 Since the differential operator $\mathcal{P}(D)$ is linear, if we set $\tilde{\gamma}=\gamma$ then our analysis using Maximum principle method on the mixed test function \eqref{mixed testf} will be identically back to that on the single test function \eqref{testf}. It is to say that the new form \eqref{mixed testf} of test functions gives us a generalization of the analysis in Section \ref{Action of the operator}. Here the essential idea is that the three above factors of Maximum principle method depend on two parameters $\gamma,\tilde{\gamma}$ instead of one, so that it is easier to control the constrains of this method on these factors to construct our desired sub--solutions and supper--solutions.

Let us first notice that the integral means exponent of $\phi_{\pm}$ \eqref{eq_mixed_sub_super_solutions} is the maximum between those of the two single test functions $\psi_0,\psi_1$. The following lemma discusses the positivity of $\psi$ (hence of $\phi_{\pm}$)
\begin{lem}\label{lem positivity}
	If $\beta(\gamma)<\beta(\tilde{\gamma})$ and $\psi_1>0$ on $\D$ then there is $r_0<1$ such that $\psi>0$ in the annulus $r_0<r<1$.
\end{lem}
\begin{proof} We already have that $\psi$ is positive in a neighborhood of $z=1$ since $\psi_0, \psi_1$ are standard test functions. For $z$ bounded away from $1$, there are $M_0, M_1>0$ such that
	\begin{equation}
	\vert \psi_0 \vert\leq M_0 (1-z\bar{z})^{-\beta(\gamma)},\,\, M_1 (1-z\bar{z})^{-\beta(\tilde{\gamma})}\leq \psi_1.
	\end{equation}
	The inequality $\beta(\gamma)<\beta(\tilde{\gamma})$ shows that in the $r\rightarrow1$ limit,
	\begin{equation}
	(1-z\bar{z})^{-\beta(\gamma)}\lessapprox(1-z\bar{z})^{-\beta(\tilde{\gamma})}.
	\end{equation}
	Therefore $\psi_1$ dominates $\psi_0$ and then $\psi$ is positive near the unit circle $r=1$.
\end{proof}
\begin{rem}\label{rem positivity}
	Condition $\beta(\gamma)<\beta(\tilde{\gamma})$ is equivalent to $\tilde{\gamma}<\gamma<\tilde{\gamma}'$ or $\tilde{\gamma} '<\gamma<\tilde{\gamma}$.
\end{rem}

We now study the action of the differential operator $\mathcal{P}(D)$ \eqref{eq1zz} on the following logarithmic modification of the mixed test function \eqref{mixed testf}
\begin{equation}
	\psi l_\delta=\psi_0 l_\delta+\psi_1 l_\delta.
\end{equation}
Since the operator $\mathcal{P}(D)$ is linear, the sign of $\mathcal{P}(D)(\psi l_\delta)$ near the unit circle $r=1$ is determined by comparing the most singular terms of $\mathcal{P}(D)(\psi_0 l_\delta)$ and $\mathcal{P}(D)(\psi_1 l_\delta)$. Let us study the three following cases:\\

{\it Case I: u is bounded away from 0.} The most singular terms of $\mathcal{P}(D)(\psi_0 l_\delta)$ and $\mathcal{P}(D)(\psi_1 l_\delta)$ as $r\to 1$ are respectively their logarithmic ones. Therefore the sign of $\mathcal{P}(D)(\psi l_\delta)$ near the unit circle $r=1$ is that of
\begin{equation}\label{log term mixed}
	-\frac{\psi_0}{u}\frac{2\delta z \bar{z}}{[-\log(1-z\bar{z})]}-\frac{\psi_1}{u}\frac{2\delta z \bar{z}}{[-\log(1-z\bar{z})]}=-\frac{\psi}{u}\frac{2\delta z \bar{z}}{[-\log(1-z\bar{z})]}.
\end{equation}

{\it Case II: $1-z\bar{z}\leq u^{1/(2-\epsilon)}$, where $0<\epsilon<2$.} The analysis in Section \ref{Action of the operator} shows that in this case the logarithmic terms are still the most singular terms of $\mathcal{P}(D)(\psi_0 l_\delta)$ and $\mathcal{P}(D)(\psi_1 l_\delta)$ as $r\to 1$. Therefore, the sign of $\mathcal{P}(D)(\psi l_\delta)$ near the unit circle $r=1$ is that of \eqref{log term mixed}.\\

{\it Case III: $u^{1/(2-\epsilon)}<1-z\bar{z}$, where $0<\epsilon<2$.} This inequality together with $1-z\bar{z}\leq 2u^{1/2}$ imply that $-\log(1-z\bar{z})\approx -\log(u)$ as $u\to 0$ (i.e $z\to 1$). One therefore has the following approximating expressions
\begin{equation}
	 \psi_1\frac{1-z\bar{z}}{u}\leq \psi_1\frac{u^{1/2}}{u},\psi_1\frac{(1-z\bar{z})^2}{u^2}\leq \psi_1\frac{1}{u}
	 ,\frac{\psi_1}{u}\frac{1}{[-\log(1-z\bar{z})]}\approx\psi_1\frac{1}{u(-\log u)}
\end{equation}
Obviously, $\psi_1/u$ dominates all terms of $\mathcal{P}(D)(\psi_1 l_\delta)$. Let us recall that in the $u\to 0$ limit
\begin{equation}
\psi_0\approx (1-r)^{-\beta(\gamma)}u^\gamma,\, \psi_1\approx (1-r)^{-\beta(\tilde{\gamma})}u^{\tilde{\gamma}}.
\end{equation}
Then one has
\begin{align}
	\psi_1\approx (1-z\bar{z})^{-\beta(\tilde{\gamma})}u^{\tilde{\gamma}}&=(1-z\bar{z})^{-\beta(\gamma)}u^{\gamma}(1-z\bar{z})^{-\beta(\tilde{\gamma})+\beta(\gamma)}u^{\tilde{\gamma}-\gamma}\\
	&\approx \psi_0 (1-z\bar{z})^{-\beta(\tilde{\gamma})+\beta(\gamma)}u^{\tilde{\gamma}-\gamma}.\nonumber
\end{align}
We assume that $\beta(\gamma)<\beta(\tilde{\gamma})$ then
\begin{equation}
	\psi_0 (1-z\bar{z})^{-\beta(\tilde{\gamma})+\beta(\gamma)}u^{\tilde{\gamma}-\gamma}\leq 	\psi_0 u^{\frac{1}{2-\epsilon}[-\beta(\tilde{\gamma})+2\tilde{\gamma}+\beta(\gamma)-2\gamma-\epsilon(\tilde{\gamma}-\gamma)]}
\end{equation}
With the additional assumption $-\beta(\gamma)+2\gamma<-\beta(\tilde{\gamma})+2\tilde{\gamma}$, one can take $\epsilon$ sufficiently small to have that the exponent of $u$ in the right hand side is positive. Then
\begin{equation}
		\psi_1\frac{1}{u}\lessapprox \psi_0\frac{1}{u(-\log u)}\approx\frac{\psi_0}{u}\frac{1}{[-\log(1-z\bar{z})]}.
\end{equation}
As a consequence, $\mathcal{P}(D)(\psi_0 l_\delta)$ dominates $\mathcal{P}(D)(\psi_1 l_\delta)$ near $z=1$. Therefore the sign of $\mathcal{P}(D)(\psi l_\delta)$ is given by that of $\mathcal{P}(D)(\psi_0 l_\delta)$. The analysis in Section \ref{Action of the operator} follows by concluding that if $C(\gamma)=0$ then the sign of $\mathcal{P}(D)(\psi l_\delta)$ is opposite to the sign of $\delta$, otherwise, one can take an appropriate sign to $\delta$ to make $\mathcal{P}(D)(\psi l_\delta)$ have the same sign as that of $C(\gamma)$.\\

The study of the above three cases takes us to the following
\begin{prop}\label{prop mixed testf}
Let $\psi$ be defined by \eqref{mixed testf} with $\gamma, \tilde{\gamma}$ such that
\begin{equation}\label{cond_gamma_gamma_tilde}
	\beta(\gamma)<\beta(\tilde{\gamma}),\,\,-\beta(\gamma)+2\gamma<-\beta(\tilde{\gamma})+2\tilde{\gamma}.
\end{equation}
Then
\begin{enumerate}[before=\leavevmode,label=\upshape(\alph*),ref=\thethm (\Roman*)]
	\item If $\gamma=\gamma_0$ then there exists $\delta_1>0,\delta_2<0$ and $r_0<1$ such that in the annulus $r_0<r<1$, $
	\mathcal{P}(D)[\psi(z,\bar{z})l_{\delta_1}(z\bar{z})]<0$ and $\mathcal{P}(D)[\psi(z,\bar{z})l_{\delta_2}(z\bar{z})]>0.$
	\item If $C(\gamma)<0$ then there exists $\delta>0$ and $r_0<1$ such that in the annulus $r_0<r<1$, $\mathcal{P}(D)[\psi(z,\bar{z})l_{\delta}(z\bar{z})]<0$.
	\item If $C(\gamma)>0$ then there exists $\delta<0$ and $r_0<1$ such that in the annulus $r_0<r<1$, $\mathcal{P}(D)[\psi(z,\bar{z})l_{\delta}(z\bar{z})]>0$.
\end{enumerate}
\end{prop}
\begin{rem}\label{rem_cond_gamma_gamma_tilde}
	Condition $\beta(\gamma)<\beta(\tilde{\gamma})$, $-\beta(\gamma)+2\gamma<-\beta(\tilde{\gamma})+2\tilde{\gamma}$ is equivalent to
	 $$\tilde{\gamma}'<\gamma<\min(\tilde{\gamma},\tilde{\gamma}'+2/\kappa).$$
\end{rem}

Remark \ref{rem_cond_gamma_gamma_tilde} shows that $\tilde{\gamma}'<\tilde{\gamma}$ is necessary for that the inequalities \eqref{cond_gamma_gamma_tilde} hold. Let us now recall an observation made in Section \ref{sec_normalized-testf}: Among standard test functions \eqref{testf}, those with $\gamma=\gamma_1-n, n\in \N$ and $g_0$ defined by \eqref{eq_gamma1-n} are the only that may have $\gamma'<\gamma$. Namely, in the domain located below the line of slope 1: $q=p-[\kappa^2(1+4n)^2-16]/32\kappa$, one has $\gamma'=\gamma_1'+n<\gamma_1-n=\gamma$. It implies that $\tilde{\gamma}$ in Proposition \ref{prop mixed testf} belongs to the set $\{\gamma_1-n: n\in \N\}$ and for $\tilde{\gamma}=\gamma_1-n$, the domain of applicability of this proposition is the half-plane below the line $q=p-[\kappa^2(1+4n)^2-16]/32\kappa$. In Proposition \ref{prop mixed testf}, let us set $\tilde{\gamma}=\gamma_1$ and state the following

\begin{prop}\label{prop mixed testf_gamma_1}
	Let
	\begin{equation}\label{mixed testf gamma1}
	\psi(z,\bar{z}):=\psi_0+\psi_1=(1-z\bar{z})^{-\beta(\gamma)}u^{\gamma}g_0(u)+(1-z\bar{z})^{-\beta(\gamma_1)}u^{\gamma_1}.
	\end{equation}
	where $\gamma_1$ is defined by \eqref{eq def gamma_1} and $\psi_0$ is a standard test function \eqref{testf} such that
	\begin{equation}\label{cond_prop_mixed testf}
	\beta(\gamma)<\beta(\gamma_1),-\beta(\gamma)+2\gamma<-\beta(\gamma_1)+2\gamma_1.
	\end{equation}
	Then
	\begin{enumerate}[before=\leavevmode,label=\upshape(\alph*),ref=\thethm (\Roman*)]
		\item If $\gamma=\gamma_0$ then there exists $\delta_1>0,\delta_2<0$ and $r_0<1$ such that in the annulus $r_0<r<1$, $
		\mathcal{P}(D)[\psi(z,\bar{z})l_{\delta_1}(z\bar{z})]<0$ and $\mathcal{P}(D)[\psi(z,\bar{z})l_{\delta_2}(z\bar{z})]>0.$
		\item If $C(\gamma)<0$ then there exists $\delta>0$ and $r_0<1$ such that in the annulus $r_0<r<1$, $\mathcal{P}(D)[\psi(z,\bar{z})l_{\delta}(z\bar{z})]<0$.
		\item If $C(\gamma)>0$ then there exists $\delta<0$ and $r_0<1$ such that in the annulus $r_0<r<1$, $\mathcal{P}(D)[\psi(z,\bar{z})l_{\delta}(z\bar{z})]>0$.
	\end{enumerate}
\end{prop}
\begin{rem}\label{rem sign}
	Condition $\beta(\gamma)<\beta(\gamma_1),-\beta(\gamma)+2\gamma<-\beta(\gamma_1)+2\gamma_1$ is equivalent to $$\gamma_1 '<\gamma<\min(\gamma_1,\gamma'_1+2/\kappa).$$
\end{rem}
Lemma \ref{lem positivity} shows that the condition \eqref{cond_prop_mixed testf} in Proposition \ref{prop mixed testf_gamma_1} includes the positivity of $\psi$ near the unit circle $r=1$.\\

Let us notice that, in \cite{DHLZ2018}, the authors used the idea of Proposition \ref{prop mixed testf_gamma_1}(a) to construct sub--solutions and super--solutions of the operator \eqref{eq1zz} and then proved Theorem \ref{spec B-D-Z}.

\subsection{Proof of Theorem \ref{thm_main_theorem}}
Let us construct mixed sub--solutions to $\mathcal{P}(D)$ and use Maximum principle method to estimate the average generalized integral means spectrum of whole--plane $\SLE$ in the domain $\mathfrak{D}_1\cap \mathcal{S}$. Recall that $\mathfrak{D}_1$ is the domain corresponding to the phase $\beta_1$ in the conjecture about the generalized spectrum $\beta(p,q)$ and $\mathcal{S}$ is the sector situated to the left of the line $\Delta_0$ and below the line $\Delta_1$, where both $\gamma_0$ and $\gamma_1$, so both $\beta_{0}$ and $\beta_{1}$, are well defined.\\

We first introduce the geometric elements that will be involved in our study and their characteristic algebraic expressions.\\
- {\it The left branch of the green parabola $\mathcal{G}$, $\gamma'_1=\gamma_0$.} One has $\gamma'_1<\gamma_0$ if and only if $(p,q)$ is located to the right of this line.\\
- {\it The right branch of the green parabola $\mathcal{G}$, $\gamma_0=\gamma'_1+2/\kappa$.} One has $\gamma_0<\gamma'_1+2/\kappa$ $ (\gamma_0>\gamma'_1+2/\kappa)$ if and only if $(p,q)$ is located to the left (right) of this line.\\
- {\it The right part starting at $Q_0$ of the blue quartic branch $\mathcal{Q}$, $\gamma'_1<\gamma_0$ and $\beta(\gamma_0)-2\gamma_0-1=\beta(\gamma_1)$.} One has $\beta(\gamma_0)-2\gamma_0-1<\beta(\gamma_1)$ if and only if $(p,q)$ is located to the right of this line.\\
- {\it The middle part $T_1T_0$ of the red parabola $\mathcal{R}$, $\gamma_0=\gamma_1$.} One has $\gamma_0<\gamma_1$ if and only if $(p,q)$ is located below this line.\\
- {\it The vertical line $D'_0$, $\gamma_0=-1/2$.} One has $\gamma_0>-1/2$ if and only if $(p,q)$ is located to the right of this line.\\
- {\it The line $D_1$ of slope 1 passing through $P_0$, $q-p=-(\kappa^2-16)/32\kappa$; $\gamma'_1=\gamma_1$.} One has $\gamma'_1<\gamma_1$ if and only if $(p,q)$ is located below this line.\\
- {\it The line $D_3$ of slope 1 passing through $Q_0$, $q-p=-1-\kappa/2$; $\gamma'_1=-1/2$.} One has $\gamma'_1<-1/2$ if and only if $(p,q)$ is located below this line.\\
- {\it The line $D_4$ of slope 1 passing through $Q'_0$, $q-p=-(2+\kappa)(4+\kappa)/2\kappa$; $\gamma'_1+2/\kappa=-1/2$.} One has $\gamma'_1+2/\kappa<-1/2$ if and only if $(p,q)$ is located below this line.\\

One should notice that the vertical line $D'_0$ passes through the point of concurrency $Q_0$ of the left branch of the green parabola, the blue quartic and the line $D_3$ as well as the intersection point $Q'_0$ of the right branch of the green parabola and the line $D_4$ (see Figure \ref{fig_involved_lines}).

\begin{figure}[h]
	\centering
	\begin{subfigure}{0.5\textwidth}
		\centering
		\includegraphics[width=0.9\linewidth]{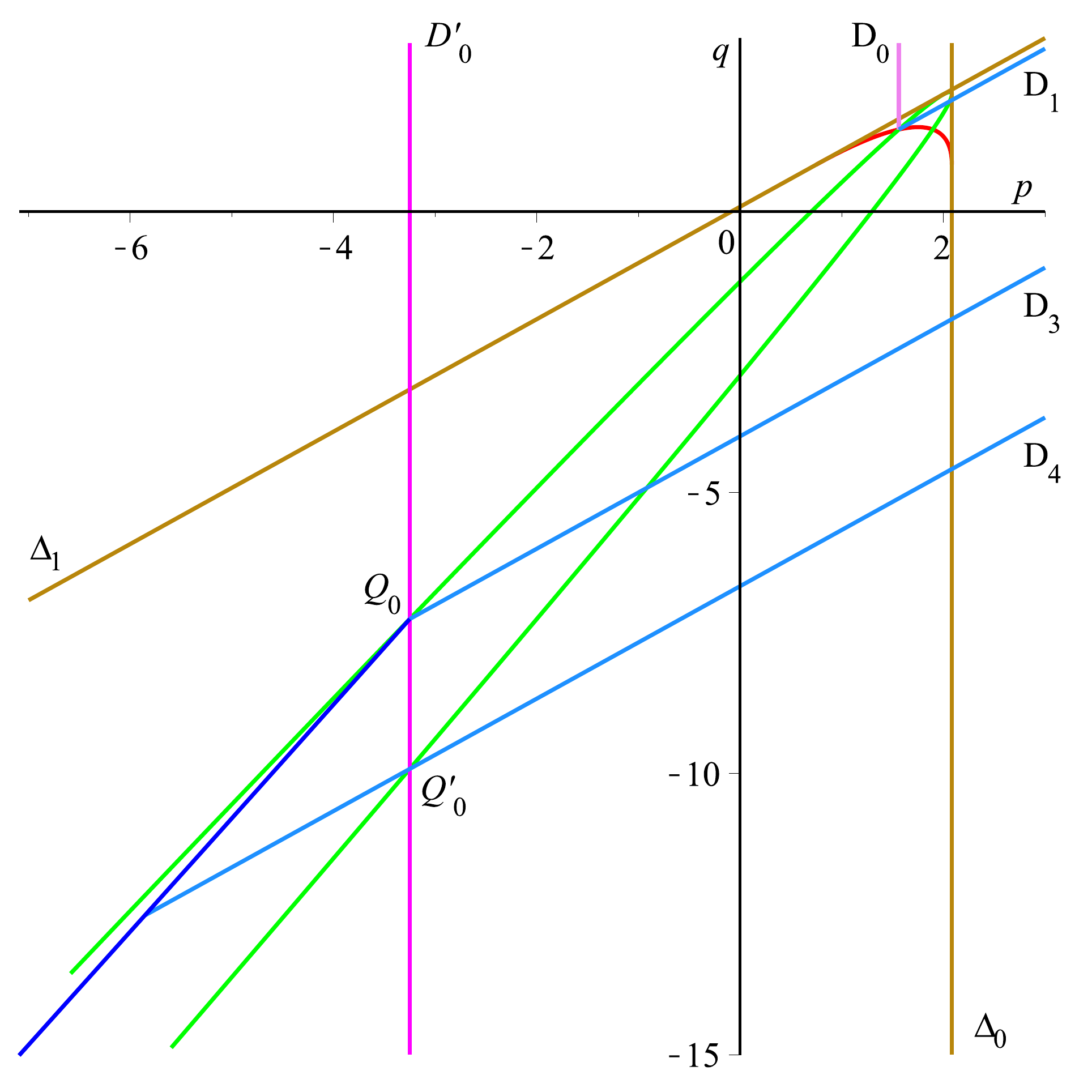}
		\caption{Large scale}
		\label{fig_involved_lines}
	\end{subfigure}%
	\begin{subfigure}{0.5\textwidth}
		\centering
		\includegraphics[width=0.9\linewidth]{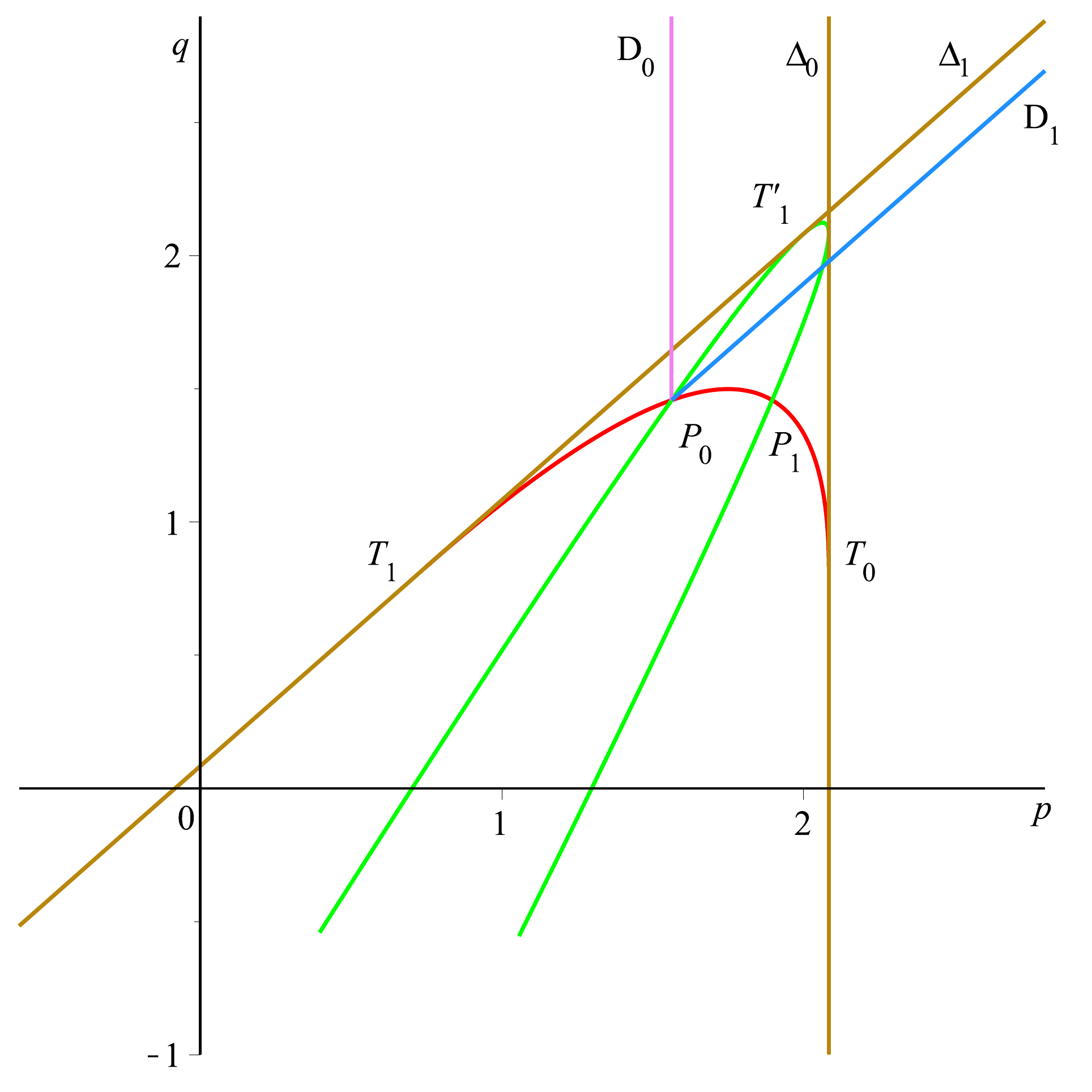}
		\caption{A zoom in the $(\Delta_0,\Delta_1)$ sector}
	\end{subfigure}
	\caption{{\it Involved lines.}}
	\label{whpl2}
\end{figure}

The above lines partition the domain $\mathfrak{D}_1\cap\mathcal{S}$ into four zones (see Figure \ref{fig_4_zones}):\\
- {\it Zone I:} closed domain situated to the right of the vertical line $D'_0$ below the point $Q_0$, to the right of the left branch of the green parabola above the point $Q_0$, below the middle part $T_1T_0$ of the red parabola and to the left of the right branch of the green parabola.\\
- {\it Zone II:} open domain situated to the right of the blue quartic up to point $Q_0$, to the left of the vertical line $D'_0$ down to point $Q'_0$, to the left of the right branch of the green parabola.\\
- {\it Zone III:} open domain situated to the right of the right branch of the green parabola up to point $Q'_0$, below the line $D_4$, to the left of the vertical line $\Delta_0$.\\
- {\it Zone IV:} closed domain whose boundary is the right branch of the green parabola between points $Q'_0$ and $P_1$, followed by the part $P_1P_0$ of the red parabola, followed by the line $D_1$ between $P_0$ and the intersection of $D_1$ and $\Delta_0$, followed by the vertical line $\Delta_0$ down to intersection point of $\Delta_0$ and $D_4$, followed by the line $D_4$ up to $Q'_0$.
\begin{figure}[h]
	\centering
	\includegraphics[width=0.5\linewidth]{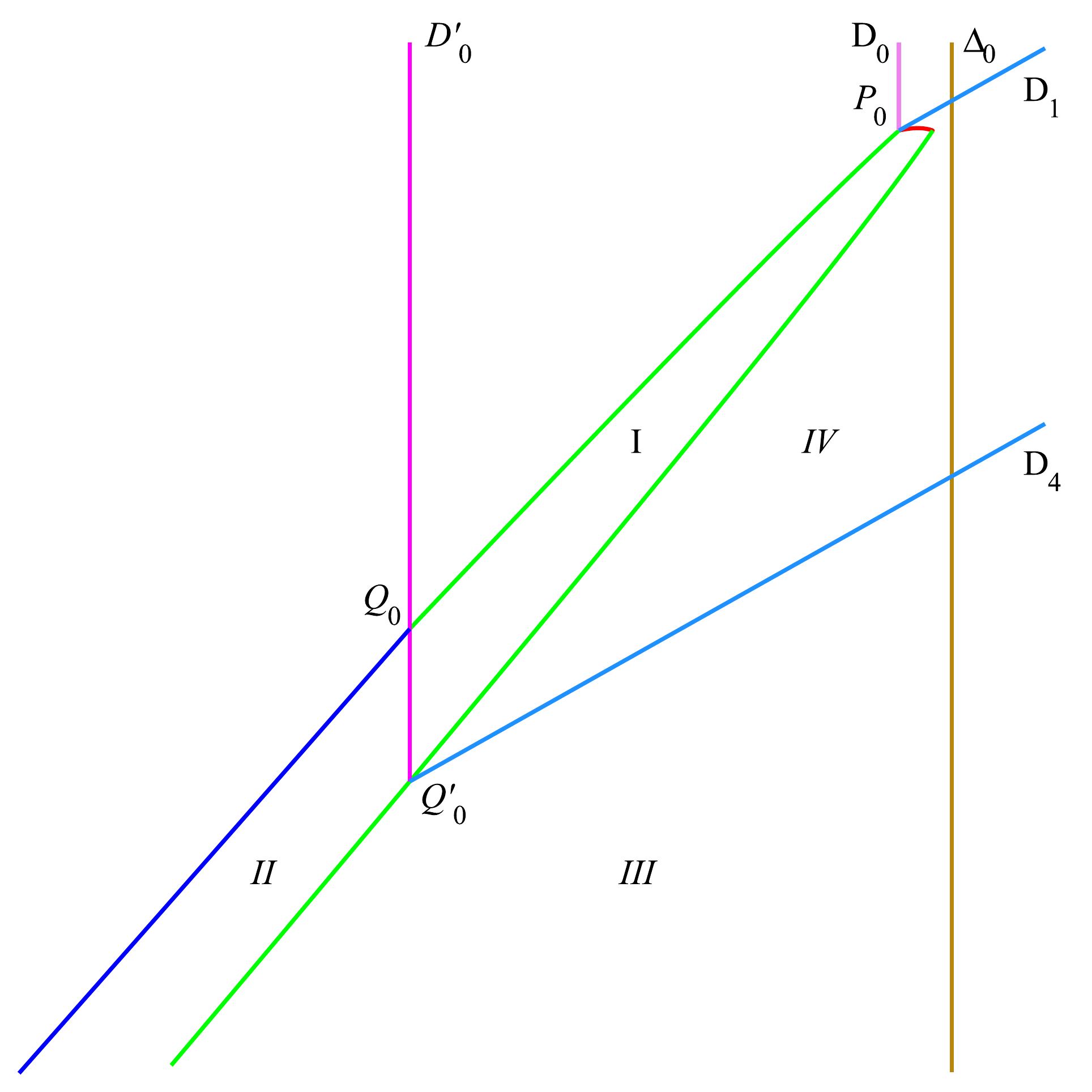}
	\caption{{\it Domain $\mathfrak{D}_1\cap\mathcal{S}$ partitioned into four zones.}}
	\label{fig_4_zones}
\end{figure}

\begin{lem}\label{upper bound}
	The average generalized integral means spectrum $\beta(p,q)$ of whole--plane $\SLE$ is bounded above as $\beta(p,q)\leq \beta_1(p,q)$ in the domain $\mathfrak{D}_1\cap \mathcal{S}$.
\end{lem}

\begin{proof}	 
We will construct test functions $\psi$ of the form \eqref{mixed testf gamma1} such that $C(\gamma)<0$, its integral mean exponent is $\beta(\gamma_1)$ and Proposition \ref{prop mixed testf_gamma_1} is applicable. Then their logarithmic modifications $\psi l_\delta, \delta>0$ are given as sub-solutions to the differential operator $\mathcal{P}(D)$. To be precise, let us detail the conditions that these functions are asked to verify. First, we need $\gamma<\gamma_{lin}$ since we want $\psi_0(z,\bar{z})$ to be a standard test function. Second, the integral means exponent of $\psi$ is given by $\beta(\gamma_1)$. Equivalently, the integral means exponent of $\psi_0$ is smaller than $\beta(\gamma_1)$. Third, the test function need to be positive. Lemma \ref{lem positivity} states that inequality $\beta(\gamma)<\beta(\gamma_1)$ is a sufficient condition for that. Finally, in order to construct sub-solutions by making use of Proposition \ref{prop mixed testf_gamma_1}, one needs  $\beta(\gamma)<\beta(\gamma_1),-\beta(\gamma)+2\gamma<-\beta(\gamma_1)+2\gamma_1$ and $C(\gamma)<0$. Remark \ref{rem sign} implies that the double inequality $\beta(\gamma)<\beta(\gamma_1),-\beta(\gamma)+2\gamma<-\beta(\gamma_1)+2\gamma_1$ is equivalent to $\gamma_1 '<\gamma<\min(\gamma_1,\gamma'_1+2/\kappa)$. One also notices that $\gamma^+_0 >{\gamma^+_0}'$, so $\gamma<\gamma^+_0$  because of the constraint $\gamma<\gamma'$. Therefore, the inequality $C(\gamma)<0$ is equivalent to $\gamma<\gamma_0$. The construction of sub-solutions is in fact to take appropriate values of $\gamma$ that verify all above conditions. Namely,
\begin{align}
	&\gamma'_1<\gamma<\min\bigg(\gamma_0,\gamma_1,\gamma'_1+\frac{2}{\kappa},\gamma_{lin}\bigg) \label{sub cond 1} \\
&\beta(\gamma_1)> \text{ integral means exponent of }\psi_0. \label{sub cond 2}
\end{align}

 We now go into the construction of sub-solutions in the four zones defined above.

	{\it Zone I.} This zone is characterized by 
	\begin{equation}
		\max\bigg(-\frac{1}{2},\gamma_1'\bigg)<\gamma_0<\min\bigg(\gamma_1'+\frac{2}{\kappa},\gamma_1\bigg).
	\end{equation}
	Let us take $\gamma$ such that $\max(-1/2,\gamma_1')<\gamma<\min(\gamma_0,\gamma_{lin})$. Obviously, $\gamma$ verifies inequalities \eqref{sub cond 1}.  Moreover, the integral means exponent of $\psi_0$ is given by $\beta(\gamma)$ which is smaller than $\beta(\gamma_1)$ by Remark \ref{rem positivity}.
	
	{\it Zone II.} This zone is characterized by
	\begin{equation}
	\gamma_1'<\gamma_0<\min\bigg(-\frac{1}{2},\gamma_1'+\frac{2}{\kappa}\bigg)<\gamma_1
	\end{equation}
	Let us take $\gamma$ such that $\gamma<\gamma_0$ and sufficiently closed to $\gamma_0$. Then inequalities \eqref{sub cond 1} are verified. We also notice that the integral means exponent of $\psi_0$ is given by $\beta(\gamma)-2\gamma+1$ that can be taken closed to $\beta(\gamma_0)-2\gamma_0+1$, while the last is smaller than $\beta(\gamma_1)$ in our current domain of consideration.
	
	{\it Zone III.} This zone is characterized by
	\begin{equation}
	\gamma_1'<\gamma_1'+\frac{2}{\kappa}<\min\bigg(-\frac{1}{2},\gamma_0\bigg)<\gamma_1.
	\end{equation}
	Let us take $\gamma$ such that $\gamma<\gamma'_1+2/\kappa$ and sufficiently closed to $\gamma'_1+2/\kappa$. Then inequalities \eqref{sub cond 1} are verified. The integral means exponent of $\psi_0$ is given by $\beta(\gamma)-2\gamma+1$ that can be taken closed to $\beta(\gamma'_1+2/\kappa)-2(\gamma'_1+2/\kappa)+1$. One notices that
	\begin{equation}
		\beta(\gamma'_1+2/\kappa)-2(\gamma'_1+2/\kappa)+1-\beta(\gamma_1)=-2\gamma_1-1<0.
	\end{equation}
	Therefore the integral means exponent of $\psi_0$ is smaller than $\beta(\gamma_1)$.
	
		{\it Zone IV.} This zone is characterized by
		\begin{equation}
			\max\bigg(-\frac{1}{2},\gamma_1'\bigg)<\min\bigg(\gamma_1'+\frac{2}{\kappa},\gamma_1\bigg)<\gamma_0.
		\end{equation}
		Let us take $\gamma$ such that $\max(-1/2,\gamma_1')<\gamma<\min(\gamma_1'+2/\kappa,\gamma_1,\gamma_{lin})$. As above, inequalities \eqref{sub cond 1} are fulfilled. The integral means exponent of $\psi_0$ is given by $\beta(\gamma)$ which is smaller than $\beta(\gamma_1)$ by Remark \ref{rem positivity}.
		
		In summary, it is possible to find a value of $\gamma$ that fulfills conditions \eqref{sub cond 1}, \eqref{sub cond 2} in each zone. Thank to Lemma \ref{lem positivity} and Proposition \ref{prop mixed testf_gamma_1}, one can construct sub-solutions for the differential operator $\mathcal{P}(D)$ such that their integral means exponent is $\beta(\gamma_1)$ as
		\begin{equation}\label{mixed sub-solution gamma1}
			\phi_{-}(z,\bar{z})=[(1-z\bar{z})^{-\beta(\gamma)}u^{\gamma}g_0(u)+(1-z\bar{z})^{-\beta(\gamma_1)}u^{\gamma_1}](-\log(1-z\bar{z}))^\delta, \delta>0.
		\end{equation}
		Therefore, the average generalized integral means spectrum $\beta(p,q)$ of whole--plane $\SLE$ is bounded above by $\beta(\gamma_1)$ in the domain $\mathfrak{D}_1\cap\mathcal{S}$ by means of the maximum principle.
\end{proof}

{\bf Proof of Theorem \ref{thm_main_theorem}.} From Proposition \ref{lower bound} we have that
\begin{equation}
	\beta_1(p,q)\leq \beta(p,q),\,\,\text{in }\mathfrak{D}_1\cap\mathcal{I}.
\end{equation}
On the other hand, since $\mathfrak{D}_1\cap\mathcal{I}\subset \mathfrak{D}_1\cap\mathcal{S}$, Lemma \ref{upper bound} implies that
\begin{equation}
\beta(p,q)\leq \beta_1(p,q),\,\,\text{in }\mathfrak{D}_1\cap\mathcal{I}.
\end{equation}
Therefore, the average generalized integral means spectrum $\beta(p,q)$ of whole--plane $\SLE$ is given by $\beta_1$ in $\mathfrak{D}_1\cap\mathcal{I}$. This finishes the proof of Theorem \ref{thm_main_theorem}.

\begin{rem}
	In this section, we have constructed sub--solutions of the form \eqref{mixed sub-solution gamma1} for the differential operator $\mathcal{P}(D)$ and applied the maximum principle to extend the validity domain of the conjecture in \cite{DHLZ2018} to the domain $\mathfrak{D}_1\cap \mathcal{S}$. In the rest of the $(p,q)$--plane, Proposition \ref{upper bound} gives us an upper bound $\beta(p,q)\leq \beta_1(p,q)$. One may ask if it is possible to use the same method as above to prove the sharpness of this estimation. Namely, if we can construct super-solutions of the form
	\begin{equation}\label{mixed super-solution gamma1}
	\phi_{+}(z,\bar{z})=[(1-z\bar{z})^{-\beta(\gamma)}u^{\gamma}g_0(u)+(1-z\bar{z})^{-\beta(\gamma_1)}u^{\gamma_1}](-\log(1-z\bar{z}))^\delta, \delta<0,
	\end{equation}
	in the rest of the parameter plane, then the generalized spectrum $\beta(p,q)$ is bounded from below by $\beta_1(p,q)$. Hence we arrive at $\beta(p,q)= \beta_1(p,q)$ and so extend the validity zone of the conjecture. Unfortunately, the answer is that it is impossible. To see that, let us try constructing super-solutions \eqref{mixed super-solution gamma1} on $\mathfrak{D}_1\cap\mathcal{S}$.
	
	We will construct mixed test functions \eqref{mixed testf gamma1} that verify the conditions in the proof of Theorem \ref{thm_main_theorem}, except that the inequality $C(\gamma)<0$ is replaced by $C(\gamma)>0$. The construction of super-solutions is in fact to take the values of $\gamma$ such that
	\begin{align}
		&\gamma_0<\gamma<\min\bigg(\gamma_0^{+},\gamma_1,\gamma'_1+\frac{2}{\kappa},\gamma_{lin}\bigg) \label{super cond 1},\\
		&\beta(\gamma_1)> \text{ integral means exponent of }\psi_0. \label{super cond 2}
	\end{align}
	We have that $\beta(\gamma_1)>\max(\beta(\gamma_0),\beta(\gamma_0)-2\gamma_0-1)$ in $\mathfrak{D}_1$. On the other hand, the integral means exponent of $\psi_0$ is given by $\beta(\gamma)$ if $\gamma\geq 1/2$ and by $\beta(\gamma)-2\gamma-1$ if $\gamma<-1/2$. So if we take $\gamma$ sufficiently closed to $\gamma_0$ then inequality \eqref{super cond 2} holds thank to the continuity of the function $\beta$. It is now sufficient to look for conditions under which there exists $\gamma$ satisfying the inequalities \eqref{super cond 1}. These conditions are
	
	\begin{equation}
	\gamma_0<\min\bigg(\gamma_1,\gamma'_1+\frac{2}{\kappa},\gamma_{lin}\bigg)
	\end{equation}
	This is geometrically translated into: $(p,q)$ belongs to the domain $\mathfrak{D}$ mentioned in Theorem \ref{spec B-D-Z}.
	Therefore, we conclude that one can construct super-solutions \eqref{mixed super-solution gamma1} to the differential operator $\mathcal{P}(D)$ in $\mathfrak{D}_1\cap\mathcal{S}$ if and only if $(p,q)$ belongs to $\mathfrak{D}$. Recall that $\mathfrak{D}$ is the domain where the spectrum has already been computed in Theorem \ref{spec B-D-Z}. Therefore, no extension of the validity zone has been made.
\end{rem}

\section*{Acknowledgement}
I would like to express my gratitude to Michel Zinsmeister who provided continuous supports and encouragements throughout this work.
I wish to thank Bertrand Duplantier for insightful conversations.
I am also grateful to the scientific and administrative teams at Institut de Math\'ematiques de Bordeaux (IMB), especially to the Analysis group, for their supports during this work.

\bibliographystyle{amsplain}
\bibliography{bib}
\nocite{*}
\end{document}